\documentclass[final,10pt,leqno]{amsart}
\usepackage{amsmath,amsfonts,amssymb,amscd,verbatim}
\usepackage[margin=0.4in]{geometry}
\usepackage{fullpage}
\usepackage{color}
\usepackage{dsfont}
\usepackage{tikz}
\usepackage{tikz-cd}
\usepackage{mathrsfs}
\usepackage{mathtools}
\usepackage{bm}
\usepackage[normalem]{ulem}
\usepackage[labelfont=rm]{subcaption}
\usepackage{hyperref}

\makeatletter
\@namedef{subjclassname@2020}{\textup{2020} Mathematics Subject Classification}
\makeatother

\usepackage{enumerate}

\usepackage{bm}

\makeatletter
\newcommand{\leqnomode}{\tagsleft@true}
\newcommand{\reqnomode}{\tagsleft@false}
\makeatother

\numberwithin{equation}{section}

\theoremstyle{definition}
\newtheorem{theorem}{Theorem} 
\newtheorem*{theorem*}{Theorem}
\numberwithin{theorem}{section}

\newtheorem{corollary}[theorem]{Corollary} 
\newtheorem{lemma}[theorem]{Lemma}
\newtheorem{proposition}[theorem]{Proposition}

\newtheorem{hypothesis}[theorem]{Hypothesis} 
\newtheorem{remark}[theorem]{Remark}
\newtheorem{definition}[theorem]{Definition} 
\newtheorem{example}[theorem]{Example}

\DeclareMathOperator\Aut{Aut}

\DeclareMathOperator\GKdim{GKdim}

\DeclareMathOperator\id{id}
\DeclareMathOperator\im{im}

\DeclareMathOperator\Maxspec{Maxspec}
\DeclareMathOperator\orb{orb}
\DeclareMathOperator\ord{ord}

\DeclareMathOperator\Spec{Spec}

\DeclareMathOperator\supp{supp}

\renewcommand{\Bbb}{\mathbb}

\newcommand\NN{\mathbb N}

\newcommand\RR{\mathbb R}
\newcommand\ZZ{\mathbb Z}

\newcommand\cB{\mathcal B}

\newcommand\cI{\mathcal I}

\newcommand\cO{\mathcal O}

\newcommand\cX{\mathcal X}

\newcommand\frm{\mathfrak m}
\newcommand\frn{\mathfrak n}

\newcommand\fsl{\mathfrak{sl}}

\newcommand\ba{\mathbf a}
\newcommand\bd{\mathbf d}
\newcommand\be{\mathbf e}
\newcommand\bp{\mathbf p}
\newcommand\bq{\mathbf q}
\newcommand\br{\mathbf r}
\newcommand\bs{\mathbf s}
\newcommand\bt{\mathbf t}
\newcommand\bu{\mathbf u}
\newcommand\bv{\mathbf v}

\newcommand\bH{\mathbf H}
\newcommand\bJ{\mathbf J}
\newcommand\bK{\mathbf K}
\newcommand\bL{\mathbf L}
\newcommand\bM{\mathbf M}
\newcommand\bN{\mathbf N}

\newcommand{\bsigma}{{\bm{\sigma}}}
\newcommand{\balpha}{{\bm{\alpha}}}
\newcommand{\bbeta}{{\bm{\beta}}}
\newcommand{\bgamma}{{\bm{\gamma}}}
\newcommand{\bdelta}{{\bm{\delta}}}
\newcommand{\btau}{{\bm{\tau}}}
\newcommand{\bpi}{{\bm{\pi}}}
\newcommand{\bphi}{{\bm{\varphi}}}
\newcommand\bzero{\mathbf 0}

\newcommand\cnt{\mathcal Z}
\newcommand\css{\mathcal S}

\newcommand\inv{^{-1}}
\newcommand\iso{\cong}
\newcommand\kk{\mathds{k}}
\newcommand\tensor{\otimes}

\newcommand{\sto}{\ensuremath{\rightarrow}}
\renewcommand{\to}{\ensuremath{\longrightarrow}}

\mathchardef\mhyphen="2D

\newcommand{\wmod}{\mathrm{\mhyphen wmod}}
\newcommand{\wmodO}{\mathrm{\mhyphen wmod_{\cO}}}

\newcommand\grp[1]{{\langle #1 \rangle}}
\newcommand\restrict[1]{\raisebox{-.3ex}{$|$}_{#1}}

\definecolor{darkgreen}{RGB}{55,138,0}
\definecolor{orange}{RGB}{255,120,0}
\definecolor{Gray}{RGB}{180,180,180}

\title{Higher rank Bell--Rogalski algebras}

\author[Gaddis]{Jason Gaddis}
\address{(Gaddis) Department of Mathematics,  Miami University, Oxford, Ohio 45056} 
\email{gaddisj@miamioh.edu}

\author[Rosso]{Daniele Rosso}
\address{(Rosso) Department of Mathematics and Actuarial Science, Indiana University Northwest, Gary, IN 46408} 
\email{drosso@iu.edu}

\author[Won]{Robert Won}
\address{(Won) Department of Mathematics, The George Washington University, Washington, DC 20052}
\email{robertwon@gwu.edu}


\begin{document}

\begin{abstract}
We generalize a construction of Bell and Rogalski to realize new examples of $\ZZ^n$-graded simple rings. This construction also generalizes TGWAs of type $(A_1)^n$. In addition to considering basic properties of these algebras, we provide a classification of weight modules in the setting of torsion-free orbits, study their (twisted) tensor products, and provide a simplicity criterion.
\end{abstract}

\maketitle

\section{Introduction}

The Weyl algebra is the prototypical example of an infinite-dimensional simple ring. More generally, the setting of $\ZZ$-graded rings has provided many examples of simple rings. For example, the quotient of $U(\fsl_2)$ by the ideal generated by $(\Omega+1)$, where $\Omega$ is the Casimir element, is a simple ring \cite{joseph,St2}. Jordan has given a simplicity criterion for skew Laurent rings \cite{jsimp1}. 

The examples above all fall within the class of \emph{generalized Weyl algebras} (GWAs). Simplicity in this setting has been studied by various authors including Bavula \cite{B3} and Jordan \cite{Jprim}.  Hartwig and \"{O}inert proved a simplicity criterion for \emph{twisted generalized Weyl algebras} (TGWAs) of Cartan type $(A_1)^n$ \cite{HO}. 
These TGWAs are $\ZZ^n$-graded generalizations of GWAs.
The goal of this work is to provide further examples of higher rank simple rings with a $\ZZ^n$-grading.

In \cite{BR}, Bell and Rogalski defined a class of $\ZZ$-graded algebras which generalize GWAs. These algebras were also studied in \cite{GR}.
In Section \ref{sec.defn}, we extend this construction to a more general class of algebras,
which we call \emph{BR algebras (of rank $n$)}. Bell and Rogalski's definition corresponds to the case where $n = 1$.
We also prove several basic properties for BR algebras which generalize results that are largely known in the case $n=1$: 
\begin{itemize}
    \item Every TGWA of type $(A_1)^n$ is a BR algebra (Theorem \ref{thm.tgwa}). 
    \item Certain invariant rings of BR algebras are also BR algebras (Theorem \ref{thm.fixed}).
    \item Gelfand--Kirillov dimension behaves as expected for BR algebras (Theorem \ref{thm.gkdim}).
\end{itemize}

In Section \ref{sec.weight} we study weight modules over BR algebras, simultaneously extending some of our previous results in rank $n=1$ from \cite{GRW1} as well as known results about (T)GWAs. Specifically, the main result is Theorem \ref{thm.weight}, which gives a complete classification of simple weight modules supported on torsion-free orbits. 

A canonical way of producing higher rank simple rings is through taking tensor products. It is well-known that the class of GWAs (of arbitrary rank) is closed under taking tensor products \cite{B1}. By \cite[Theorem 2.16]{GR}, the class of TGWAs which satisfy certain mild conditions is closed under tensor products. In \cite{GR3}, this was extended to $\ZZ^n$-graded twisted tensor products of TGWAs (under similar hypotheses). In Section \ref{sec.twtensor}, we show that BR algebras are also closed under twisted tensor products which preserve the natural grading on both tensor factors (Theorem \ref{thm.twtensor}). This leads to new examples of $\ZZ^n$-graded simple rings.

In Section \ref{sec.simplicity}, we give a simplicity criterion for BR algebras of rank $n$ (Theorem \ref{thm.HBRsimple}). This reduces to Bell and Rogalski's simplicity criterion in the rank $1$ case \cite[Proposition 2.18 (3)]{BR} (see also \cite[Proposition 2.3]{GR}).
However, for $n>1$, this criterion is not 
easy to check. We conclude with some ideas that may lead eventually to a better, more checkable set of conditions.

The diagram below presents the classes of algebras discussed, where arrows represent inclusion:
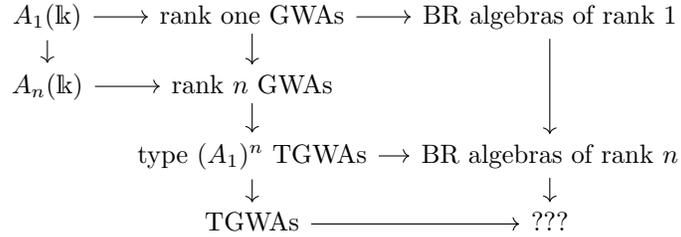
\[ 
\begin{tikzcd}[sep=small]
A_1(\kk) \ar[r] \ar[d] & \text{rank one GWAs} \ar[r] \ar[d] & \text{BR algebras of rank 1} \ar[dd] \\
A_n(\kk) \ar[r] & \text{rank $n$ GWAs} \ar[d] &  \\
& \text{type $(A_1)^n$ TGWAs} \ar[r] \ar[d] & \text{BR algebras of rank $n$} \ar[d] \\
& \text{TGWAs} \ar[r] & ???
\end{tikzcd}
\]
A future goal is to understand the missing corner. That is: is there a construction that properly generalizes TGWAs and BR algebras? Such a construction should also generalize, in a reasonable way, the ring theoretic and/or representation theoretic properties of both classes.

\subsection*{Acknowledgements}
D. R. was supported in this work by a Grant-In-Aid of research and a Summer Faculty Fellowship from Indiana University Northwest, and by an AMS--Simons Research Enhancement Grant for PUI Faculty.
R. W. was supported by Simons Foundation grant \#961085.

\section{Definitions and initial results}
\label{sec.defn}

Throughout this paper, let $\kk$ be a field. All algebras in this paper will be associative unital $\kk$-algebras. Let $[n]$ denote the set $\{1,\hdots,n\}$. A matrix $\bp = (p_{ik}) \in M_n(\kk^\times)$ is \emph{multiplicatively antisymmetric} if $p_{ik}=p_{ki}\inv$ for all $i, k \in [n]$.

Let $G$ be an abelian monoid. An algebra $A$ is \emph{$G$-graded} if there is a $\kk$-vector space decomposition $A=\bigoplus_{g \in G} A_g$ such that $A_g \cdot A_h \subseteq A_{gh}$ for all $g,h \in G$. For $g \in G$, an element $a \in A_g$ is called \emph{homogeneous} of degree $g$. For a $G$-graded Ore domain $A$, we write $Q_{\mathrm{gr}}(A)$ for the \emph{graded quotient ring of $A$}, the localization of $A$ at all nonzero homogeneous elements.  Throughout this paper, all $A$-modules will be left $A$-modules. An $A$-module $M$ is called \emph{$G$-graded} if there is a $\kk$-vector space decomposition $M = \bigoplus_{g \in G} M_g$ such that $A_g \cdot M_h  \subseteq M_{gh}$ for all $g,h \in G$. A (left/right/two-sided) ideal of $A$ is \emph{$G$-graded} if it is generated by homogeneous elements. In this paper, we focus almost exclusively on the case $G=\ZZ^n$.

\begin{definition}
Let $R$ be an algebra, let $\bsigma=(\sigma_1,\hdots,\sigma_n)$ be an $n$-tuple of commuting algebra automorphisms of $R$, and let $\bp = (p_{ik}) \in M_n(\kk^\times)$ be multiplicatively antisymmetric. The corresponding \emph{iterated skew Laurent extension over $R$}, denoted $R_{\bp}[\bt^{\pm 1};\bsigma]$, is generated over $R$ by $t_1^{\pm},\hdots,t_n^{\pm}$ subject to the relations
\begin{align*}
t_i^{\pm 1} r &= \sigma_i^{\pm 1}(r) t_i^{\pm} \quad \text{and} \quad
t_kt_i = p_{ik} t_i t_k
\qquad \text{for all $i, k \in [n]$, $i\neq k$, $r \in R$.}
\end{align*}
Note that $R_{\bp}[\bt^{\pm 1};\bsigma]$ is $\ZZ^n$-graded by setting $\deg t_i = \be_i = (0, \dots, 0, \stackrel{i}{1}, 0, \dots, 0)$ for all $i \in [n]$ and $\deg r = \bzero$ for all $r \in R$. In fact, if $R$ is $\Gamma$-graded by an abelian monoid $\Gamma$, then $R_{\bp}[\bt^{\pm 1};\bsigma]$ is $(\Gamma \times \ZZ^n)$-graded by setting $\deg t_i  = (0, \be_i)$ for all $i \in [n]$ and $\deg r = (\gamma, \bzero)$ for $r \in R_{\gamma}$.
\end{definition}

When $n=1$, there is only a single automorphism $\sigma$ and the matrix $\bp$ is irrelevant. 
The corresponding skew Laurent extension over $R$ is denoted $R[t^{\pm 1}; \sigma]$ in this case.

Let $R$ be an algebra, $\sigma \in \Aut(R)$, and let $H$ and $J$ be ideals of $R$. For each $k\in \ZZ$, set
\[
I_{\sigma}^{(k)}(H,J) = \begin{cases}
	J\sigma(J)\cdots \sigma^{k-1}(J) & \text{if $k > 0$,} \\
	R & \text{if $k = 0$,}\\
	\sigma\inv(H)\sigma^{-2}(H)\cdots\sigma^{k}(H) & \text{if $k < 0$.} 
	\end{cases}
\]

\begin{definition}\label{def.brd}
Let $R$ be an algebra and $n$ a positive integer. Let $\bsigma=(\sigma_1,\hdots,\sigma_n)$ be an $n$-tuple of commuting algebra automorphisms of $R$. Let $\bp \in M_n(\kk^\times)$ be multiplicatively antisymmetric. Let $\bH = (H_1,\hdots,H_n)$ and $\bJ = (J_1,\hdots,J_n)$ be $n$-tuples of two-sided ideals of $R$. For each $i \in [n]$ and $k \in \ZZ$, let $I_i^{(k)} := I_{\sigma_i}^{(k)}(H_i,J_i)$. The tuple $(R,\bt,\bsigma,\bp,\bH,\bJ)$ is a \emph{Bell--Rogalski datum} (BRD) if 
\begin{align}\label{eq.comm1}
\sigma_i(H_k)=H_k \quad\text{and}\quad \sigma_i(J_k)=J_k \quad\text{for $i \neq k \in [n]$},
\end{align}
\begin{align}
\label{eq.comm2}
H_iH_k = H_kH_i, \quad H_iJ_k = J_kH_i, \quad \text{and} \quad J_iJ_k = J_kJ_i \quad \text{for $i\neq k \in [n]$},
\end{align}
and $I_{i}^{(k)} \neq 0$ for all $i \in [n]$ and all $k \in \ZZ$.
\end{definition}

By applying $\sigma_i$ to equation~\eqref{eq.comm2}, it follows that $H_k$ and $J_k$ commute with $\sigma_i^\ell(H_i)$ and $\sigma_i^\ell(J_i)$ for every  $i \neq k \in [n]$ and every integer $\ell$. 
Given the setup of Definition \ref{def.brd} and $\balpha = (\alpha_1, \dots, \alpha_n) \in \ZZ^n$, let $I^{(\balpha)} = I_1^{(\alpha_1)} \cdots I_n^{(\alpha_n)}$ and $\bt^{\balpha} = t_1^{\alpha_1} \cdots t_n^{\alpha_n}$. We are now ready to define a higher-rank version of Bell--Rogalski algebras \cite{BR,GRW1}. 

\begin{definition}\label{defn.BRalg}
Given a BRD $(R,\bt,\bsigma,\bp,\bH,\bJ)$, the corresponding \emph{Bell--Rogalski (BR) algebra (of rank $n$)} is defined as 
\[ R_{\bp}(\bt,\bsigma,\bH,\bJ) = \bigoplus_{\balpha \in \ZZ^n} I^{(\balpha)} \bt^\balpha \subset R_{\bp}[\bt^{\pm 1};\bsigma].\]
\end{definition}

The definition above in rank one reduces to the construction of Bell and Rogalski, which was used to classify $\ZZ$-graded simple rings up to graded Morita equivalence \cite{BR}. These were also studied by the authors in \cite{GRW1}. 
Throughout, we will suppress ``of rank $n$" unless the parameter $n$ is relevant to the discussion.

\begin{lemma}\label{lem.domain}
Let $B$ be a BR algebra over an algebra $R$. Then $B$ is a domain if and only if $R$ is a domain.
\end{lemma}
\begin{proof}
If $B$ is a domain, then so is $B_{\bzero}=R$. Conversely, if $R$ is a domain, then the iterated skew Laurent extension $R_{\bp}[\bt^{\pm 1}; \bsigma]$ is a domain, hence so is the subalgebra $B$.
\end{proof}

The following should be compared to \cite[Lemma 2.10(1)]{BR} and \cite[Lemma 2.1(1)]{GRW1}. Note that every algebra $R$ admits a trivial grading, so $\Gamma$ can be taken to be trivial below.

\begin{lemma}\label{lem.BRgen}
Let $B = R_{\bp}(\bt, \bsigma, \bH, \bJ)$ be a BR algebra of rank $n$. Suppose that $R$ is $\Gamma$-graded by an abelian monoid $\Gamma$. Then $B$ is a $(\Gamma \times \ZZ^n)$-graded subring of $R_{\bp}[\bt^{\pm 1}; \bsigma]$. Moreover, $B$ is generated as an algebra by $B_\bzero = R$, $B_{\be_i} = J_it_i$, and $B_{-\be_i} = \sigma_i\inv(H_i)t_i\inv$, over all $i \in [n]$. In particular, if $R$ is finitely generated as an algebra and each $J_i$ and $H_i$ is finitely generated as an ideal of $R$, then $B$ is finitely generated as an algebra.
\end{lemma}
\begin{proof}
Let $\balpha,\bbeta \in \ZZ^n$. By \eqref{eq.comm1} and \eqref{eq.comm2} applied to each coordinate, we have
\[
I^{(\balpha)}\bt^{\balpha} I^{(\bbeta)}\bt^{\bbeta}
= I^{(\balpha)}\bsigma^{\balpha}\left(I^{(\bbeta)}\right)\bt^{\balpha+\bbeta}
= I^{(\balpha+\bbeta)}\bt^{\balpha+\bbeta}.
\]
Hence, $B$ is indeed a subring of $R_{\bp}[\bt^{\pm 1}; \bsigma]$ which inherits the natural $\Gamma \times \ZZ^n$-grading.

If $\balpha \in \ZZ^n$, then
\[
B_{\balpha} = I^{(\balpha)} \bt^{\balpha} = \prod_{k = 1}^n I_k^{(\alpha_k)} t_k^{\alpha_k} = \prod_{k=1}^n B_{\alpha_k \be_k}.
\]
Hence, it suffices to show that $B_\bzero$, $B_{\be_k}$, and $B_{-\be_k}$ generate each $B_{\alpha \be_k}$, where $\alpha \in \ZZ$. If $\alpha = 0$, then $B_{\alpha \be_k} = B_\bzero$. If $\alpha > 0$, then since
\[
I_k^{(\alpha \be_k)} t_k^{\alpha} = J_k \sigma_k(J_k) \cdots \sigma_k^{\alpha-1}(J_k) t_k^{\alpha} = (J_k t_k) \cdots (J_k t_k),
\]
then $B_{\alpha \be_k} = (B_{\be_k})^{\alpha}$. Similarly, if $\alpha < 0$, then $B_{\alpha \be_k} = (B_{-\be_k})^{-\alpha}$. The result follows.
\end{proof}

\begin{lemma}\label{lem.iter}
Let $B$ be a BR algebra of rank $n>1$. 
Then $B$ can be expressed as a BR algebra of rank $1$ over a BR algebra of rank $n - 1$.
\end{lemma}
\begin{proof}
Let $B = R_{\bp}(\bt, \bsigma, \bH, \bJ)$ be as above. Let $\bp' \in M_{n-1}(\kk^\times)$ be the submatrix of $\bp$ consisting of rows and columns $1,\hdots,n-1$. Note that $\bp'$ is again multiplicatively antisymmetric. Let $\bsigma'=(\sigma_1,\hdots,\sigma_{n-1})$, let $\bt'=(t_1, \dots, t_{n-1})$, let $\bH'=(H_1, \dots, H_{n-1})$, and let $\bJ'=(J_1, \dots, J_{n-1})$. Set
\[
S = R_{\bp'}(\bt', \bsigma', \bH',\bJ') \cong \bigoplus_{\substack{\balpha \in \ZZ^n \\ \text{with } \alpha_n = 0}} I^{(\balpha)} \bt^\balpha,
\]
so that $S$ is a rank $n-1$ BR algebra with its natural $\ZZ^{n-1}$-grading.

We extend $\sigma_n$ to an automorphism of $S$ by setting $\sigma_n(t_i)=p_{in}t_i$ for $i<n$. 
Set $\widehat{H}=H_nS$ and $\widehat{J}=J_nS$.
Note that $\widehat{H}$ and $\widehat{J}$ are two-sided ideals of $S$ since $H_n$ and $J_n$ are two-sided ideals of $R$ which are fixed by $\sigma_i$ for all $i < n$. Then it is clear that 
$S(t_n,\sigma_n,\widehat{H},\widehat{J})$ is a $\ZZ^{n-1} \times \ZZ = \ZZ^n$-graded ring which is graded isomorphic to $B$.
\end{proof}

The class of BR algebras is closed under tensor products. This is not difficult to show. We give a more general argument in Section \ref{sec.twtensor} where we show that the class of BR algebras is, in fact, closed under certain types of \emph{twisted} tensor products.

\subsection{Relation to TGWAs}

The next definition is a special type of \emph{twisted generalized Weyl algebra} (TGWA). We do not give the full definition of a TGWA here, but recall a particular presentation of this type due to Futorny and Hartwig \cite[Theorem 4.1]{FH1}.

\begin{definition}\label{def.tgwa}
Let $R$ be an algebra and $n$ a positive integer. Let $\bsigma=(\sigma_1,\hdots,\sigma_n)$ be an $n$-tuple of commuting automorphisms and let $\ba=(a_1,\hdots,a_n)$ be an $n$-tuple of central, non-zero divisors in $R$ and suppose there exist, for $i \neq k$, scalars $\gamma_{ik} \in \kk^\times$ such that $\sigma_i(a_k)=\gamma_{ik}a_k$.
For $\mu_{ik} \in \kk^\times$ with $i\neq k$ satisfying 
\begin{align}\label{eq.Acons}
\mu_{ik}\mu_{ki} = \gamma_{ik}\gamma_{ki} \text{ for all $i \neq k$},
\end{align}
the corresponding \emph{twisted generalized Weyl algebra (TGWA) of type $(A_1)^n$}, denoted $A_\mu(R,\bsigma,\ba)$, is generated over $R$ by the $2n$ indeterminantes $X_1^{\pm},\hdots,X_n^{\pm}$ modulo the two-sided ideal generated by the following set of elements:
\begin{align*}
&X_i^{\pm} r - \sigma_i^{\pm 1}(r) X_i^{\pm}, \quad
X_i^-X_i^+ - a_i, \quad X_i^+X_i^- - \sigma_i(a_i),
&\text{for all $r \in R$ and $i \in [n]$,} \\
&X_i^+ X_k^- - \mu_{ik} X_k^- X_i^+, \quad
X_i^+ X_k^+ - \gamma_{ik}\mu_{ik}\inv X_k^+ X_i^+, \quad
X_k^- X_i^- - \gamma_{ik} \mu_{ki}\inv X_i^- X_k^-,
 &\text{for all $i,k \in [n]$, $i \neq k$.}
\end{align*}
\end{definition}

We say a BR algebra $R_\bp(\bt,\bsigma,\bH,\bJ)$ has \emph{scalar units} if $\sigma_i(u)=u$ for all $u \in R^\times$ and all $i \in [n]$. This hypothesis appears in the classification of TGWAs over polynomial rings \cite[Definition 2.5]{HR1} and applies to the most relevant examples of BR algebras.

\begin{theorem}\label{thm.tgwa}
Let $R$ be a commutative algebra. 
\begin{enumerate}
    \item Let $B=R_\bp(\bt,\bsigma,\bH,\bJ)$ be a BR algebra of rank $n$ with scalar units. If $H_i,J_i$ are principal for each $i$, then $B$ is graded isomorphic to a consistent TGWA of type $(A_1)^n$ over $R$.

    \item If $A=A_\mu(R,\bsigma,\ba)$ is a TGWA of type $(A_1)^n$ over $R$, then there exists a BR algebra $B$ such that $A \iso B$ as graded algebras.
\end{enumerate}
\end{theorem}
\begin{proof}
(1)
Suppose that for all $k=1,\ldots,n$ we have $J_k=(j_k)$ and $H_k=(h_k)$. Since, for $i\neq k$, $\sigma_i$ leaves both ideals $J_k$, $H_k$ invariant, we have that the image of a generator needs to generate the ideal, hence
\[
\sigma_i(j_k) = u_{ik}j_k \qquad
\sigma_i(h_k) = v_{ik}h_k
\]
for all $i \neq k$ for some units $u_{ik},v_{ik}\in R^\times$. Since $B$ has scalar units, then $u_{ik},v_{ik}$ are invariant under $\sigma_\ell$ for all $\ell \in [n]$.

Set $a_k = \sigma_k\inv(h_kj_k)$. For $i \neq k$, set $\gamma_{ik}=u_{ik}v_{ik}$, and $\mu_{ik}=u_{ki}v_{ik}p_{ik}$. Then
\[ \sigma_i(a_k) =\sigma_i(\sigma_k\inv(h_kj_k))=\sigma_k\inv(\sigma_i(h_kj_k))=\sigma_k\inv(u_{ik}v_{ik}h_kj_k) = u_{ik}v_{ik} a_k = \gamma_{ik} a_k.\]
Define $A=A_\mu(R,\bsigma,\ba)$. We claim $A \iso B$.

Define a $\ZZ^n$-graded map $\phi:A \to B$ by $\phi(X_i^+)=j_it_i$ and $\phi(X_i^-)=\sigma_i\inv(h_i)t_i\inv$ for $i \in [n]$, and $\phi(r)=r$ for all $r \in R$. Then for $i,k \in [n]$, $i \neq k$, and $r \in R$, we have 
\begin{align*}
\phi(X_i^+)r - \phi(\sigma(r))\phi(X_i^+)
    &= (j_it_i)r - \sigma(r)(j_it_i) = 0 \\
\phi(X_i^-)r - \phi(\sigma(r))\phi(X_i^-)
    &= (\sigma_i\inv(h_i)t_i\inv) r - \sigma(r)(\sigma_i\inv(h_i)t_i\inv) = 0 \\
\phi(X_i^-)\phi(X_i^+) - \phi(a_i)
    &= (\sigma_i\inv(h_i)t_i\inv)(j_it_i) - \phi(a_i) = 0 \\
\phi(X_i^+)\phi(X_i^-) - \phi(\sigma(a_i))
    &= (j_it_i)(\sigma_i\inv(h_i)t_i\inv) - \phi(\sigma(a_i)) = 0 \\
\phi(X_i^+)\phi(X_k^-) - \mu_{ik}\phi(X_k^-)\phi(X_i^+)
        &= (j_it_i)(\sigma_k\inv(h_k)t_k\inv) 
            - u_{ki}v_{ik}p_{ik} (\sigma_k\inv(h_k)t_k\inv)(j_it_i) \\
        &= v_{ik}p_{ik} \sigma_k\inv(h_k)(j_it_k\inv) t_i
            - u_{ki}v_{ik}p_{ik} (\sigma_k\inv(h_k)t_k\inv)(j_it_i) 
            = 0 \\
\phi(X_i^+)\phi(X_k^+) - \gamma_{ik}\mu_{ik}\inv X_k^+X_i^+
        &= (j_it_i)(j_kt_k) - (u_{ik}v_{ik})(u_{ki}v_{ik}p_{ik})\inv (j_kt_k)(j_it_i) \\
        &= p_{ki}u_{ik} j_k(j_it_k)t_i - p_{ki}u_{ik}u_{ki}\inv (j_kt_k)(j_it_i) = 0.
\end{align*}
Finally,
\begin{align*}
\phi(X_k^-)\phi(X_i^-)&-\gamma_{ik}\mu_{ki}\inv \phi(X_i^-)\phi(X_k^-) \\
        &= (\sigma_k\inv(h_k)t_k\inv)(\sigma_i\inv(h_i)t_i\inv)
            - (u_{ik}v_{ik})(u_{ik}v_{ki}p_{ki})\inv (\sigma_i\inv(h_i)t_i\inv)(\sigma_k\inv(h_k) t_k\inv) \\
        &= p_{ik} v_{ki}\inv \sigma_i\inv(h_i)(\sigma_k\inv(h_k)t_i\inv) t_k\inv 
            - p_{ik} v_{ki}\inv v_{ik} (\sigma_i\inv(h_i)t_i\inv)(\sigma_k\inv(h_k) t_k\inv) 
        = 0.
\end{align*}
Note that this TGWA is consistent since
\[
\mu_{ik}\mu_{ki} = (u_{ki}v_{ik}p_{ik})(u_{ik}v_{ki}p_{ki})
    = (u_{ik}v_{ik})(u_{ki}v_{ki})
    = \gamma_{ik}\gamma_{ki}.
\]

Since $B$ is generated as an algebra by $R$ and $\{j_i t_i, \sigma\inv(h_i)t_i\inv \mid i \in [n]\}$, we can define an inverse map $\psi:B \to A$ by $\psi(j_it_i) = X_i^+$ and $\psi(\sigma_i\inv(h_i)t_i\inv) = X_i^-$ for all $i \in [n]$, and $\psi(r) = r$ for all $r \in R$. Since $B$ is a subalgebra of the skew Laurent ring $R_{\bp}[\bt^{\pm 1}; \bsigma]$, we need only verify that $\psi$ preserves the relations in the skew Laurent ring (between the generators of $B$). Note for all $i,k \in [n]$ with $i\neq k$ and $r \in R$, we have
\begin{align*}
\psi(j_it_i)\psi(r) - \psi(\sigma_i(r))\psi(j_it_i)
    &= X_i^+ r - \sigma_i(r)X_i^+ = 0 \\
\psi(\sigma\inv(h_i)t_i\inv)\psi(r) - \psi(\sigma_i\inv(r))\psi(\sigma\inv(h_i)t_i\inv)
    &= X_i^- r - \sigma_i\inv(r)X_i^- = 0 \\
\psi(j_it_i)\psi(\sigma\inv(h_i)t_i\inv) - \psi(\sigma_i(a_i))
    &= X_i^+X_i^- - \sigma_i(a_i) = 0 \\
\psi(\sigma\inv(h_i)t_i\inv)\psi(j_it_i) - \psi(a_i)
    &= X_i^-X_i^+ - a_i = 0 \\
\psi(j_it_i)\psi(\sigma\inv(h_k)t_k\inv) - (\gamma_{ki} p_{ik})\psi(\sigma\inv(h_k)t_k\inv)\psi(\sigma_i(a_i)t_i)
    &= X_i^+X_k^- - \mu_{ik} X_k^- X_i^+ = 0 \\
\psi(j_it_i)\psi(j_k t_k) - (\gamma_{ik}\gamma_{ki}\inv p_{ki})\psi(j_k t_k)\psi(j_i t_i)
    &= X_i^+X_k^+ - \mu_{ik}\inv \gamma_{ik} X_k^+ X_i^+ = 0 \\
\psi(\sigma\inv(h_k)t_k\inv)\psi(\sigma\inv(h_i)t_i\inv) - p_{ik} \psi(\sigma\inv(h_i)t_i\inv)\psi(\sigma\inv(h_k)t_k\inv)
    &= X_k^-X_i^- - \gamma_{ik}\mu_{ki}\inv X_i^- X_k^- = 0.
\end{align*}
Hence, $\psi$ extends to an algebra homomorphism $B \to A$. Furthermore, it is clear that $\phi$ and $\psi$ are inverses.

(2) 
For $i \in [n]$, set $H_i = R$ and $J_i = (j_i)$ where $j_i=\sigma_i(a_i)$. Then $u_{ik}=\gamma_{ik}$ and $p_{ik}=\mu_{ik}\gamma_{ki}\inv$. Let $B=R_\bp(\bt,\bsigma,\bH,\bJ)$ be the BR algebra defined by this data. The above arguments show that $A \iso B$.
\end{proof}

\subsection{Invariant rings}

We now discuss invariant rings under certain automorphisms on higher rank BR algebras. This generalizes \cite[Proposition 2.13]{GRW1}. In light of the above, this also generalizes \cite[Theorem 4.5]{GR} in the case that the base ring $R$ is commutative.

\begin{lemma}\label{lem.induced}
Let $R$ be a commutative algebra. Let $B=R_\bp(\bt,\bsigma,\bH,\bJ)$ and $B'=R_{\bp'}'(\bt',\bsigma',\bH',\bJ')$ be BR algebras. Suppose $\phi:R \sto R'$ is a $\kk$-algebra homomorphism such that for each $i \in [n]$, $\phi(H_i) \subseteq H_i'$, $\phi(J_i)\subseteq J_i'$, and $\phi\sigma_i=\sigma_i'\phi$.
Then for each $\bgamma \in (\kk^{\times})^n$ there is a graded algebra morphism $\Phi_{\bgamma}:B \sto B'$ satisfying 
$\Phi_{\bgamma}(r t_i) = \gamma_i\phi(r) t_i'$ for all $r \in J_i$ and $\Phi_{\bgamma}(r t_i\inv) = \gamma_i\inv\phi(r) (t_i')\inv$ for all $r \in \sigma_i\inv(H_i)$
such that the following diagram commutes:
\[\begin{tikzcd}        
    R \arrow{r}{\phi} \arrow{d}{\iota} & R'  \arrow{d}{\iota'} \\
    B \arrow{r}{\Phi_\bgamma} & B' 
\end{tikzcd}\]
Furthermore, we have the following:
\begin{enumerate}
    \item If $\Phi_{\bgamma}$ is injective (resp. surjective), then $\phi$ is injective (resp. surjective).
    \item If $R'$ is a domain and $\phi$ is injective, then $\Phi_{\bgamma}$ is injective.
    \item If $\phi(H_i)=H_i'$ and $\phi(J_i)=J_i'$ for all $i\in[n]$, and $\phi$ is surjective, then $\Phi_{\bgamma}$ is surjective.
\end{enumerate}
\end{lemma}
\begin{proof}
Define $\Phi_{\bgamma}: R_{\bp}[\bt^{\pm 1};\bsigma] \sto R_{\bp'}'[(\bt')^{\pm 1};\bsigma']$ by $\Phi_{\bgamma}(t_i^{\pm 1}) = \gamma_i^{\pm 1} (t_i')^{\pm 1}$, and $\Phi_{\bgamma}(r) = \phi(r)$ for all $r \in R$ and all $i \in [n]$.
Then for all $r \in R$ and $i \in [n]$,
\[
\Phi_{\bgamma}(t_ir - \sigma_i(r)t_i) =
\Phi_{\bgamma}(t_i)\Phi_{\bgamma}(r)-\Phi_{\bgamma}(\sigma_i(r))\Phi_{\bgamma}(t_i)
    = \gamma_i t_i'\phi(r) - \phi(\sigma_i(r)) \gamma t_i'
    = \gamma_i (\sigma_i'\phi-\phi\sigma_i)(r)t_i'
    = 0.
\]
Hence, $\Phi_{\bgamma}$ can be extended linearly and multiplicatively to an algebra homomorphism $R_{\bp}[\bt^{\pm 1};\bsigma] \sto R_{\bp'}'[(\bt')^{\pm 1};\bsigma']$. Since $\phi(H_i) \subseteq H_i$ and $\phi(J_i) \subseteq J_i$ for all $i \in [n]$, then it is clear that $\Phi_{\bgamma}$ restricts to a homomorphism $B \sto B'$. By abuse of notation we call the restriction $\Phi_{\bgamma}$, as well. By construction, $\Phi_{\bgamma}$ respects the $\ZZ^n$-grading, so is a graded algebra map.

(1) Since $\Phi_{\bgamma}$ and $\phi$ agree on $R$, then it is clear that when $\Phi_{\bgamma}$ injective/surjective, then so is $\phi$. 

(2) Suppose $R'$ is a domain and $\phi$ is injective. Let $at^\balpha \in \ker\Phi_{\bgamma}$. Then $0=\Phi_\bgamma(at^\balpha) = \bgamma^\balpha \phi(a) (t')^\balpha$. Since $B'$ is a domain by Lemma \ref{lem.domain}, then $\phi(a)=0$. Thus, $a=0$.

(3) Suppose $\phi(H_i)=H_i'$ and $\phi(J_i)=J_i'$ for all $i\in[n]$, and $\phi$ is surjective. To show that $\Phi_{\bgamma}$ is surjective, it suffices to show that the image contains $B'_{\bzero}$ and $B_{\pm\be_i}'$ for all $i$, as these generate $B'$ as an algebra, by Lemma \ref{lem.BRgen}. If $a' \in B_{\bzero}' = R'$, since $\phi$ is surjective, there exists $a \in B_\bzero$ such that $a' = \phi(a) = \Phi_{\bgamma}(a)$. Further, if $a't_i' \in B'_{\be_i}$, then $a' \in J_i$ and so by hypothesis, there exists $a \in J_i$ such that $\phi(a) = a'$. Then $\Phi_{\bgamma}(\gamma_i\inv at_i) = a't_i'$. Analogously, $B'_{-\be_i}$ is in the image of $\Phi_{\bgamma}$, and so $\Phi_{\bgamma}$ is surjective.
\end{proof}

\begin{theorem}\label{thm.fixed}
Let $R$ be a commutative domain and let
$B=R_\bp(\bt,\bsigma,\bH,\bJ)$ be a BR algebra.
Let $\phi$ be an automorphism of $R$ such that $\phi\sigma_i=\sigma_i\phi$ for all $i \in [n]$. Let $\bgamma \in (\kk^\times)^n$ and extend $\phi$ to an automorphism $\Phi_\bgamma$ of $R_{\bp}[\bt^{\pm 1};\bsigma]$. Let $m_i$ be the order of $\gamma_i$ and suppose $\ord(\phi),m_1,\hdots,m_n$ are pairwise coprime. Set $\bq=(q_{ik})$ where $q_{ik} = p_{ik}^{m_i m_k}$, $\bs=(t_i^{m_i})$, and $\btau=(\sigma_i^{m_i})$.
\begin{enumerate}
\item
$(R_{\bp}[\bt^{\pm 1};\bsigma])^{\grp{\Phi_\gamma}} = (R^{\grp{\phi}})_{\bq}[\bs^{\pm 1};\btau]$.

\item Suppose $\phi(J_i) \subset J_i$ and $\phi(H_i) \subset H_i$ for each $i$.
For each $i$, set
\begin{align*}
    K_i &= \left(\sigma_i\inv(H_i)\sigma_i^{-2}(H_i) \cdots \sigma_i^{-m_i}(H_i) \right) \cap R^\grp{\phi} \\
    L_i &= \left( J_i\sigma_i(J_i) \cdots \sigma_i^{m_i-1}(J_i) \right) \cap R^\grp{\phi}.
\end{align*}
Then 
\[ B^\grp{\Phi_\bgamma} = (R^{\grp{\phi}})_\bq\left(\bs,\btau,\bK,\bL \right).\]
\end{enumerate}
\end{theorem}
\begin{proof}
(1) The inclusion $\supset$ is clear. It is also clear that $\Phi_\bgamma$ preserves the $\ZZ^n$-grading on $A=R_{\bp}[\bt^{\pm 1};\bsigma]$. Let $rt^\balpha \in A^{\grp{\Phi_\bgamma}}$ with $r \in R$, $\balpha \in \ZZ^n$. Since $rt^\balpha = \Phi_\bgamma(rt^\balpha) = \bgamma^\balpha\phi(r) t^\balpha$, then because $R$ (and hence $B$) is a domain by Lemma \ref{lem.domain}, we have $\phi(r) = \bgamma^{-\balpha}r$. Set $m=m_1\cdots m_n$. Then 
\[ \phi^m(r) = (\bgamma^{-\balpha})^m r = (\bgamma^m)^{-\balpha}r = r.\] 
Hence, $\left|\orb_\phi(r)\right|$ divides $m$, but $\left|\orb_\phi(r)\right|$ divides $\ord(\phi)$ by the Orbit-Stabilizer Theorem. This is a contradiction unless $m_i \mid \alpha_i$ for each $i$. Hence, $\alpha_i=\ell_i m_i$. Thus, $r \in R^\grp{\phi}$.

(2) The hypotheses imply that $\Phi_\bgamma(B) \subset B$. Hence, \[B^\grp{\Phi_\bgamma} = R_{\bp}[\bt^{\pm 1};\bsigma]^{\grp{\Phi_\bgamma}} \cap B.\]
Since $B^\grp{\Phi_\bgamma}$ inherits the $\ZZ^n$-grading on $R_{\bp}[\bt^{\pm 1};\bsigma]$ and $B$, then $(B^\grp{\Phi_\bgamma})_\bzero = R^{\grp{\phi}}$. Let $\balpha \in \ZZ^n$. If $(R_{\bp}[\bt^{\pm 1};\bsigma]^{\grp{\Phi_\bgamma}})_\balpha \neq 0$, then for each $i \in [n]$, $\alpha_i = \beta_i m_i$ for some $\beta_i \in \ZZ$. Since $B^\grp{\Phi_\bgamma} \subset B$, then the coefficient of $\bs^\bbeta$ lies in $I^{\balpha} \cap R^{\grp{\phi}}$.
\end{proof}

\subsection{GK dimension}

An automorphism $\sigma$ of an algebra $R$ is said to be \emph{locally algebraic} if for every $r$, the set $\{ \sigma^n(r) : n \in \NN\}$ is contained in a finite-dimensional subspace of $R$. The following lemma is likely well known, but because we were unable to find a reference, we provide a short proof for the benefit of the reader.

\begin{lemma}\label{lem.gkdim1}
Let $R_\bp[\bt^{\pm 1};\bsigma]$ be an iterated skew Laurent extension over $R$ of rank $n$. Suppose each $\sigma_i$ is a locally algebraic automorphism of $R$. Then
$\GKdim R_\bp[\bt^{\pm 1};\bsigma] = \GKdim R + n$.
\end{lemma}
\begin{proof}
The case $n=1$ is \cite[Proposition 1]{LMO}. We proceed inductively.

For each $k$, extend $\sigma_k$ to an automorphism of $R[t_1;\sigma_1]\cdots[t_{k-1};\sigma_{k-1}]$ by defining $\sigma_k(t_i) = p_{ik} t_i$ for all $i < k$.
Consider a homogeneous element $rt^\balpha \in S = R[t_1;\sigma_1]\cdots[t_{k-1};\sigma_{k-1}]$. 
Since $\sigma_k$ is a locally algebraic automorphism of $R$, then there is a finite-dimensional subspace $V$ of $R$ such that $\{ \sigma_k^n(r) : n \in \NN\} \subset V$. 
Choose a basis $\{v_1, v_2, \dots, v_d\}$ of $V$. 
Let $V' \subset S$ be the $\kk$-span of $\{v_1 t^\balpha, \hdots, v_d t^\balpha\}$. Since $\sigma_k$ scales $t^\balpha$, then it is clear that $\{ \sigma_k^n(rt^\balpha) : n \in \NN\} \subset V'$.

Now suppose that $s \in S$ and write $s = \sum_{\bgamma \in \Gamma} r_\bgamma t^\bgamma$ for some finite set $\Gamma \subseteq \NN^{k-1}$. 
By the above, for each $\bgamma \in \Gamma$, we have
$\{ \sigma^n(r_\bgamma t^\bgamma) : n \in \NN\} \subset V_\bgamma$ where $V_\bgamma$ is some finite-dimensional subspace of $S$. But then $\{ \sigma^n(s) : n \in \NN\} \subset \bigoplus_{\bgamma \in \Gamma} V_\bgamma$, which is finite-dimensional. Hence $\sigma_k$ is a locally algebraic automorphism of $S$ so $\GKdim S[t_{k},\sigma_{k}] = \GKdim S + 1$ by \cite[Proposition 1]{LMO}. 

It follows that each $\sigma_i$ is locally algebraic on $T=R[t_1;\sigma_1] \cdots [t_n;\sigma_n]$. Now let $X$ be the multiplicative sets generated by the $t_i$. By \cite[Theorem 2]{LMO}, 
\[ 
\GKdim R_\bp[\bt^{\pm 1};\bsigma] 
    = \GKdim TX\inv = \GKdim T = \GKdim R +n.\]
This proves the result.
\end{proof}

\begin{theorem}\label{thm.gkdim}
Let $B=R_{\bp}(\bt, \bsigma, \bH, \bJ)$ be a BR algebra of rank $n$. Suppose that each $\sigma_i$ is a locally algebraic automorphism of $R$. Then $\GKdim R\leq\GKdim B \leq \GKdim R + n$. Moreover, if each $J_i$ (resp. each $H_i$) contains a normal element which is not a zero divisor, then $\GKdim B = \GKdim R + n$. In particular, this holds if $R$ is a commutative domain.
\end{theorem}
\begin{proof}
Since $B$ is a subalgebra of a skew Laurent extension of $R$, the first statement follows directly from Lemma \ref{lem.gkdim1}.

Now suppose that each $J_i$ contains a normal element $a_i$ which is not a zero divisor. Then there are automorphisms $\tau_i$ such that $a_ir=\tau_i(r)a_i$ for all $r \in R$. Set $x_i=a_i t_i$, so $x_ir=\tau_i\sigma_i(r)x_i$ for each $i$ and for each $r \in R$. Thus, the Ore extension $A=R[x_1;\tau_1\sigma_1]\cdots[x_n;\tau_n\sigma_n]$ is a subalgebra of $B$. By \cite[Lemma 1]{LMO}, $\GKdim A \geq \GKdim R+n$. Since $A$ is a subalgebra of $B$, then this combined with part (1) gives the result. The proof in the case that each $H_i$ contains a normal element is similar.
\end{proof}

\begin{corollary}\label{cor.gkdim}
Let $A=A_\mu(R,\bsigma,\ba)$ be a TGWA of type $(A_1)^n$ over a commutative $R$. Suppose that each $\sigma_i$ is locally algebraic on $R$. Then $\GKdim A = \GKdim R + n$.
\end{corollary}
\begin{proof}
This follows from Theorems \ref{thm.tgwa} and \ref{thm.gkdim}. 
\end{proof}

\section{Weight modules}
\label{sec.weight}

Throughout this section we assume that $R$ is a commutative algebra. Let $B=R_{\bp}(\bt,\bsigma,\bH,\bJ)$ be a BR algebra of rank $n$ over $R$. Let $\Gamma$ denote the subgroup of $\Aut(R)$ generated by the $\sigma_i \in \bsigma$.

There is a natural $\ZZ^n$-action on $\Maxspec(R)$ given by
\[ \balpha \cdot \frm = \bsigma^\balpha(\frm) = \sigma_1^{\alpha_1}\sigma_2^{\alpha_2}\cdots \sigma_n^{\alpha_n}(\frm) \qquad \text{for all $\balpha \in \ZZ^n$,  $\frm \in \Maxspec(R)$}.\]
Here we use the fact that the $\sigma_i$'s all commute so that this action is well defined. We say an orbit $\cO \in \Maxspec(R)/\ZZ^n$ is \emph{torsion-free} if for any $\frm \in \cO$, $\bsigma^\balpha(\frm)=\frm$ implies that $\balpha=\bzero$. Throughout we work only in the setting of torsion-free orbits. In particular, this implies that each $\sigma_i$ has infinite order.

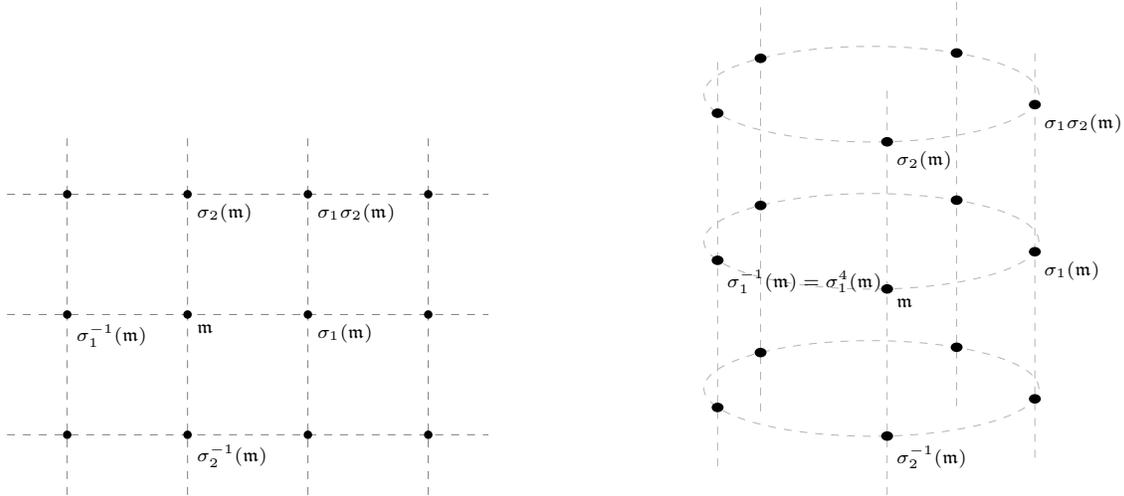
\begin{figure}[ht!]
\begin{subfigure}{.45\textwidth}
\centering
\begin{tikzpicture}[scale = .8]
\foreach \y in { -1, 0, 1}{\draw[help lines, dashed] (-3, 2*\y)--(5, 2*\y);}
\foreach \x in { -1,0,1,2}{\draw[help lines, dashed] (2*\x, -3)--(2*\x, 3);}
\foreach \y in { -1, 0, 1}{
	\foreach \x in { -1, 0, 1, 2}{
		\fill (2*\x, 2*\y) circle (2pt);
	}
}
\node[font = \scriptsize, anchor = north west] at (0, 0){\scriptsize $\mathfrak{m}$};
\node[font = \scriptsize, anchor = north west] at (2, 0){\scriptsize $\sigma_1(\mathfrak{m})$};
\node[font = \scriptsize, anchor = north west] at (-2, 0){\scriptsize $\sigma_1^{-1}(\mathfrak{m})$};
\node[font = \scriptsize, anchor = north west] at (2, 2){\scriptsize $\sigma_1\sigma_2(\mathfrak{m})$};
\node[font = \scriptsize, anchor = north west] at (0, 2){\scriptsize $\sigma_2(\mathfrak{m})$};
\node[font = \scriptsize, anchor = north west] at (0, -2){\scriptsize $\sigma_2^{-1}(\mathfrak{m})$};
\end{tikzpicture}
\caption{A torsion-free orbit of a maximal ideal $\frm$ of $R$ under the action of $\ZZ^2 = \langle \sigma_1, \sigma_2 \rangle$.}
\end{subfigure}
\begin{subfigure}{.1\textwidth}
~
\end{subfigure}
\begin{subfigure}{.4\textwidth}
\centering
\begin{tikzpicture}[scale = .9, x={(-.7cm,-.25cm)},
                    y={(.7cm,-.25cm)},
                    z={(0cm, .87cm)},
                    xscale=1.25]
  \def\r{2}       % Physical radius of cylinder
  \def\h{5}       % Physical height of cylinder

\foreach \z in {0,2.5,5}
{
  \draw[-,Gray,very thin, dashed] (\r,0,\z)
    \foreach \t in {5,10,...,360}
      {--({\r*cos(\t)},{\r*sin(\t)},\z)};
}

  \foreach \t in {1,...,5}
  {
    \def\ang{(\t-1.3)*360/5}
    \draw[-,Gray,very thin, dashed] ({\r*cos(\ang)},{\r*sin(\ang)},-1) to ({\r*cos(\ang)},{\r*sin(\ang)},6);
    \foreach \z in {0, 2.5, 5}{
    \fill ({\r*cos(\ang)}, {\r*sin(\ang)}, \z) circle (2pt);
        }
    }
  \node[font = \scriptsize, anchor = north west] at ({\r*cos(.7*360/5)}, {\r*sin(.7*360/5)}, 2.5){\scriptsize $\mathfrak{m}$};
    \node[font = \scriptsize, anchor = north west] at ({\r*cos(.7*360/5)}, {\r*sin(.7*360/5)}, 5){\scriptsize $\sigma_2(\mathfrak{m})$};
    \node[font = \scriptsize, anchor = north west] at ({\r*cos(.7*360/5)}, {\r*sin(.7*360/5)}, 0){\scriptsize $\sigma_2^{-1}(\mathfrak{m})$};
 \node[font = \scriptsize, anchor = north west] at ({\r*cos(1.7*360/5)}, {\r*sin(1.7*360/5)}, 2.5){\scriptsize $\sigma_1(\mathfrak{m})$};
  \node[font = \scriptsize, anchor = north west] at ({\r*cos(-.3*360/5)}, {\r*sin(-.3*360/5)}, 2.5){\scriptsize $\sigma_1^{-1}(\mathfrak{m}) = \sigma_1^4(\mathfrak{m})$};
   \node[font = \scriptsize, anchor = north west] at ({\r*cos(1.7*360/5)}, {\r*sin(1.7*360/5)}, 5){\scriptsize $\sigma_1\sigma_2(\mathfrak{m})$};
\end{tikzpicture}
\caption{An orbit of a maximal ideal $\frm$ of $R$ that is not torsion-free. Here, $\sigma_1^5(\frm) = \frm$.}
\end{subfigure}
\caption{We picture a maximal ideal $\frm$ as a point in $\RR^n$ and picture the action of the automorphism $\sigma_i$ as translating $\frm$ in the direction of the $i$th coordinate axis. In this section, since we assume orbits are torsion-free, the orbits can be pictured as the lattice $\ZZ^n \subseteq \RR^n$.}
\label{fig1}
\end{figure}

We say a $B$-module $M$ is a \emph{weight module} if 
\[ M = \bigoplus_{\frm \in \Maxspec(R)} M_\frm \quad\text{where}\quad M_\frm = \{ m \in M \mid \frm \cdot m = 0\}.\]
We call $M_\frm$ a \emph{weight space} for $M$. The \emph{support} of $M$ is defined as
\[ \supp(M) = \{ \frm \in \Maxspec(R) \mid M_\frm \neq 0\}.\]

The category of weight modules for $B$ is denoted by $(B,R)\wmod$. Let $\cO \in \Maxspec(R)/\ZZ^n$ be an orbit. We denote by $(B,R)\wmod_\cO$ the full subcategory of modules $M \in (B,R)\wmod$ such that $\supp(M) \subseteq \cO$. If $M \in (B,R)\wmod$, then $B_{\pm \be_i} M_{\frm} \subseteq M_{\sigma_i^{\pm 1}(\frm)}$ for each $\frm \in \supp(M)$ and each $i \in [n]$. It follows that
\[ (B,R)\wmod \simeq \prod_{\cO \in \Maxspec(R)/\ZZ^n} (B,R)\wmod_{\cO}.\]

Our goal will be to describe the simple objects of $(B,R)\wmod_\cO$ corresponding to a torsion-free orbit. Our results restrict to those of higher rank GWAs and, more generally, TGWAs of Cartan type $(A_1)^n$ \cite{MT}.
 
\begin{definition}\label{defn.ibreak}
Let $i \in [n]$. Set
\begin{align}\label{eq.css}
\css_i(B) = \Spec(R/H_iJ_i) = \{ \bp \in \Spec(R) \mid \bp \supset H_iJ_i
\}.
\end{align}
Then $\frm \in \Maxspec(M)$ is an \emph{$i$-break}, or \emph{break with respect to $i$}, if $\sigma_i(\frm) \in \css_i(B)$.
\end{definition}

The term ``$i$-break" here should not be confused with that of $j$-breaks and $h$-breaks in \cite{GRW1}.

\begin{lemma}\label{lem.invariant-breaks}
Let $\frm\in\Maxspec(M)$ be an $i$-break, and let $i\neq k\in[n]$. Then $\sigma_k^{\pm 1}(\frm)$ is also an $i$-break.
\end{lemma}
\begin{proof}
If $\sigma_i(\frm)\in \css_i(B)$, then $\sigma_i(\frm)\supset H_iJ_i$, so $\sigma_i(\sigma_k^{\pm 1}(\frm))=\sigma_k^{\pm 1}\sigma_i(\frm)\supset \sigma_k^{\pm 1}(H_iJ_i)=H_iJ_i$,
hence the result follows.
\end{proof}

\begin{proposition}\label{prop.dim-weight}
Suppose that $\cO=\ZZ^n\cdot \frm \subset \Maxspec(R)/\ZZ^n$ is a torsion-free orbit and $M\in (B,R)\wmodO$. 
\begin{enumerate}
    \item \label{dim1} If $w\in M_{\frm}$ is a weight vector and $N=B\cdot w\subseteq M$, then $\dim_{R/\bsigma^{\balpha}(\frm)}N_{\bsigma^{\balpha}(\frm)}\leq 1$ for all $\balpha\in \ZZ^n$.
    \item \label{dim2} If $M$ is simple, then 
    $\dim_{R/\bsigma^{\balpha}(\frm)}M_{\bsigma^{\balpha}(\frm)}\leq 1$ for all $\balpha\in \ZZ^n$.
\end{enumerate}
\end{proposition}
\begin{proof}
\eqref{dim1}
Since $\cO$ is a torsion-free orbit, we have $B_{\balpha} w=N_{\bsigma^{\balpha}(\frm)}$ for all $\balpha\in\ZZ^n$. Let $u, u'\in N_{\bsigma^{\balpha}(\frm)}$, then there exist $r,r'\in I^{(\balpha)}$ such that $u:=r\bt^{\balpha}\cdot w$, $u':=r'\bt^{\balpha}\cdot w$. It follows that
\[r' u-ru'= r'r\bt^{\balpha}w -rr'\bt^{\balpha}w=0,\]
hence $u$ and $u'$ are linearly dependent over $R$, and since $u', u\in N_{\bsigma^{\balpha}(\frm)}$, the linear dependence descends to show that $u$ and $u'$ are linearly dependent over $R/\bsigma^{\balpha}(\frm)$.

\eqref{dim2} For any nonzero weight vector $0\neq w\in M$, we have $B\cdot w= M$ because $M$ is simple. The result then follows from part \eqref{dim1}.
\end{proof}

\begin{lemma}\label{lem:cross-break}
Let $\cO$ be a torsion-free orbit and let $\frm\in\cO$ be an $i$-break. Let $M$ be a simple weight module. Then $B_{\be_i}M_{\frm} = 0$ and $B_{-\be_i}M_{\sigma_i(\frm)}=0$.
\end{lemma}
\begin{proof}
We prove that $B_{\be_i}M_{\frm} = 0$; the other proof is similar. If $M_{\frm} = 0$, then the statement is obvious. Now
let $0\neq w\in M_{\frm}$, and $j_i \in J_i$. Suppose that $w':=j_it_i\cdot w\neq 0$. Then for all $h_i \in H_i$ we have 
\[\sigma_i\inv(h_i)t_i\inv\cdot w'=\sigma_i\inv(h_i)t_i\inv j_it_i\cdot w=\sigma_i\inv(h_i)\sigma_i\inv(j_i)\cdot w=\sigma_i\inv(h_i j_i)\cdot w=0\]
where the last equality follows because $\frm$ is an $i$-break, which means that $\sigma_i\inv(H_i J_i) \subset \frm$.
Since $\cO$ is a torsion-free orbit, $(B w')_\frm=B_{-\be_i} w'=\sigma\inv(H_i)t_i\inv\cdot w'=0$, so $w\not\in B_{-\be_i}w'$. Hence $w\not\in B w'$ and so $0\neq B w'\subsetneq M$ is a nontrivial proper weight submodule, contradicting the simplicity of $M$.
\end{proof}

Let $\cO \in \Maxspec(R)/\ZZ^n$ be a torsion-free orbit and let $\beta_i\subseteq \cO$ be the set of $i$-breaks for some $i \in [n]$.
By Lemma \ref{lem.invariant-breaks}, there is an action of $\Gamma_i=\langle \sigma_1,\ldots,\sigma_{i-1},\sigma_{i+1},\ldots,\sigma_n\rangle\simeq \ZZ^{n-1}$ on $\beta_i$. We consider the set $\overline{\beta}_i=\beta_i/\Gamma_i$ and we think of the elements of $\overline{\beta}_i$ as hyperplanes, normal to the direction of $\be_i$, in the rank $n$ lattice given by $\cO$.
For each $i$, we define a partial order on $\cO$ by setting $\frm \prec_i \sigma_i(\frm)$ for all $\frm\in \cO$ and extending transitively. Notice that this ordering is well defined because $\cO$ is torsion-free. This partial order induces a total order on $\overline{\beta}_i$ by setting $[\frm]\prec_i[\frn]$ if there exist $\gamma_1,\gamma_2\in\Gamma_i$ with $\gamma_1(\frm)\prec_i\gamma_2(\frn)$. We define
\[
\overline{\beta}'_i = \begin{cases} \{\infty_i\} & \text{if $\overline{\beta}_i=\emptyset$} \\
\overline{\beta}_i\cup\{\infty_i\} 
    & \text{if $\overline{\beta}_i$ contains a maximal element with respect to $\prec_i$} \\
\overline{\beta}_i
    & \text{otherwise.}
\end{cases}\]
For $[\frn]\in\overline{\beta}_i'$, we let $[\frn]^-$ be the maximal element of $\overline{\beta}_i'$ such that $[\frn]^-\prec_i[\frn]$, or $[\frn]^-=-\infty_i$ if $[\frn]$ was a minimal element of $\overline{\beta}_i'$. We extend the partial ordering on $\cO$ to $\cO\cup\{\pm\infty_i\}$ in the obvious way, i.e., for all $\frm\in\cO$ we have $-\infty_i\prec_i \frm\prec_i \infty_i$. Also, by abuse of notation, for $\frm\in\cO$ and $[\frn]\in\overline{\beta}_i'$ we say that $\frm\preceq_i[\frn]$ (resp. $[\frn]\preceq_i \frm$) if there exists $\gamma\in\Gamma_i$ such that $\frm\preceq_i\gamma\frn$ (resp. $\gamma\frn\preceq_i\frm$). This denotes on which side of the hyperplane $[\frn]$ the lattice point $\frm$ sits, and in case of equality it means that $\frm\in[\frn]$ is a point on the hyperplane. We give an explicit example of picturing a torsion-free orbit and its breaks in Figure~\ref{fig2}.

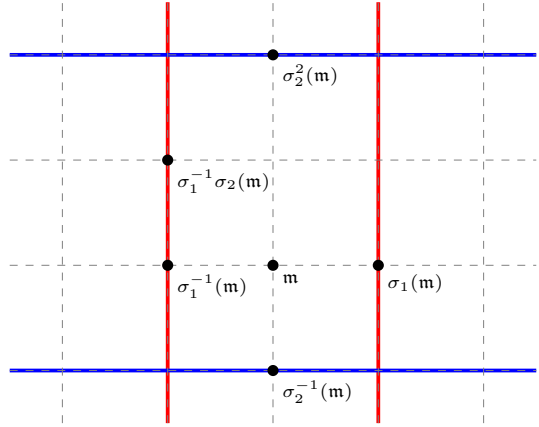
\begin{figure}[ht!]    
\begin{tikzpicture}[scale = .7]
\draw[red, line width=.5mm] (2, -3)--(2,5);
\draw[red, line width=.5mm] (-2, -3)--(-2,5);
\draw[blue, line width=.5mm] (-5, 4)--(5,4);
\draw[blue, line width=.5mm] (-5, -2)--(5,-2);
\foreach \y in { -1, 0, 1, 2}{\draw[help lines, dashed] (-5, 2*\y)--(5, 2*\y);}
\foreach \x in {-2, -1,0,1,2}{\draw[help lines, dashed] (2*\x, -3)--(2*\x, 5);}
\node[font = \scriptsize, anchor = north west] at (0, 0){\scriptsize $\mathfrak{m}$};
\fill (0, 0) circle (3pt);
\node[font = \scriptsize, anchor = north west] at (2, 0){\scriptsize $\sigma_1(\mathfrak{m})$};
\fill (2, 0) circle (3pt);
\node[font = \scriptsize, anchor = north west] at (-2, 0){\scriptsize $\sigma_1^{-1}(\mathfrak{m})$};
\fill (-2, 0) circle (3pt);
\node[font = \scriptsize, anchor = north west] at (-2, 2){\scriptsize $\sigma_1^{-1}\sigma_2(\mathfrak{m})$};
\fill (-2, 2) circle (3pt);
\node[font = \scriptsize, anchor = north west] at (0, 4){\scriptsize $\sigma_2^2(\mathfrak{m})$};
\fill (0, 4) circle (3pt);
\node[font = \scriptsize, anchor = north west] at (0, -2){\scriptsize $\sigma_2^{-1}(\mathfrak{m})$};
\fill (0, -2) circle (3pt);
\end{tikzpicture}
\caption{For the pictured orbit, $\sigma_1^{-1}(\frm)$ is a $1$-break, and so $[\sigma_1^{-1}(\frm)]=\{\sigma_1^{-1}\sigma_2^k(\frm)~|~k\in\ZZ\}$ can be pictured as the hyperplane $x = -1$ in our picture. Similarly, $\sigma_1(\frm)$ is a $1$-break, pictured on the hyperplane $x = 1$. We therefore see that $\overline{\beta}_1 = \{ [\sigma_1\inv(\frm)], [\sigma_1(\frm)] \}$ and $\overline{\beta}'_1 = \{ [\sigma_1\inv(\frm)], [\sigma_1(\frm)] , \infty_1 \}$. We also have $[\sigma_1\inv(\frm)]^- = -\infty_1$,  $[\sigma_1(\frm)]^- = [\sigma_1\inv(\frm)]$. For the $2$-breaks, we have $\overline{\beta}_2' = \{[\sigma_2\inv(\frm)], [\sigma_2^2(\frm)], \infty_2\}$. }
\label{fig2}
\end{figure}

Our analysis of simple weight modules below follows \cite{hart1} for TGWAs. 
Let $\frm$ be a maximal ideal. Define
\[ G_\frm = \{ \balpha \in \ZZ^n \mid B_{-\balpha}B_\balpha \not\subset \frm \}.\]
In \cite{hart1}, this set was called $\widetilde{G_\frm}$.

\begin{lemma}\label{lem.Gm}
Let $\frm\in\cO$, then there exists a unique $n$-tuple $([\frn_1],\ldots,[\frn_n])\in \prod_{i=1}^n\overline{\beta}'_i$ such that $[\frn_i]^-\prec_i\frm\preceq_i[\frn_i]$ for all $i$. Then 
\[G_\frm=\{\balpha\in\ZZ^n~|~[\frn_i]^-\prec_i\sigma_i^{\alpha_i}(\frm)\preceq_i[\frn_i]\quad\text{for all }i=1,\ldots,n\}.\]
\end{lemma}
\begin{proof}
Recall that $B_\balpha=I^{(\balpha)}t^{\balpha}=\prod_i I_i^{\alpha_i}t_i^{\alpha_i}$, and $B_{-\balpha}=I^{(-\balpha)}t^{-\balpha}=\prod_i I_i^{-\alpha_i}t_i^{-\alpha_i}$, then
\[ B_{-\balpha}B_\balpha=\left(\prod_i I_i^{-\alpha_i}t_i^{-\alpha_i}\right)\left(\prod_i I_i^{\alpha_i}t_i^{\alpha_i}\right)=\prod_iI_i^{-\alpha_i}\sigma_i^{-\alpha_i}\left(I_i^{\alpha_i}\right).\]
Since $\frm$ is maximal, $B_{-\balpha}B_\balpha\subset\frm$ if and only if there is $i$ such that $I_i^{-\alpha_i}\sigma_i^{-\alpha_i}\left(I_i^{\alpha_i}\right)\subset \frm$. Equivalently,  $B_{-\balpha}B_\balpha\not\subset\frm$ if and only for all $i$, $I_i^{-\alpha_i}\sigma_i^{-\alpha_i}\left(I_i^{\alpha_i}\right)\not\subset \frm$. We have
\[
I_i^{-\alpha_i}\sigma_i^{-\alpha_i}\left(I_i^{\alpha_i}\right)
    = \begin{cases}
    R & \text{ if }\alpha_i=0 \\ 
    \prod_{k=1}^{\alpha_i} \sigma_i^{-k}(H_iJ_i) & \text{ if }\alpha_i>0 \\ 
    \prod_{k=1}^{-\alpha_i} \sigma_i^{k-1}(H_iJ_i) & \text{ if }\alpha_i<0.\end{cases}\]
Let $\balpha=(\alpha_1,\ldots,\alpha_n)\in G_\frm$. If for an index $i$ we have $\alpha_i=0$, then this is the same as $R\not\subset\frm$, which is always satisfied just like the fact that $[\frn_i]^-\prec_i\sigma_i^{0}(\frm)\preceq_i[\frn_i]$ is always true. If $\alpha_i>0$, then this says that 
\begin{align*}
\prod_{k=1}^{\alpha_i} \sigma_i^{-k}(H_iJ_i)\not\subset\frm & \iff \sigma_i^{-k}(H_iJ_i)\not\subset\frm \qquad \forall~k=1,\ldots,\alpha_i \\
& \iff H_iJ_i\not\subset\sigma_i(\sigma_i^{k-1}(\frm)) \qquad \forall~k=1,\ldots,\alpha_i \\
& \iff\sigma_i^{k-1}(\frm) \text{ is not an $i$-break }\qquad \forall ~k=1,\ldots,\alpha_i \\ 
& \iff \sigma_i^{k-1}(\frm)\prec_i[n_i] \qquad \forall~k=1,\ldots,\alpha_i \\ 
& \iff \sigma_i^{k}(\frm)\preceq_i[n_i]\qquad \forall~k=0,\ldots,\alpha_i.
\end{align*}
Finally, if $\alpha_i<0$, then this says that 
\begin{align*}\prod_{k=1}^{-\alpha_i} \sigma_i^{k-1}(H_iJ_i)\not\subset\frm & \iff \sigma_i^{k-1}(H_iJ_i)\not\subset\frm \qquad \forall~k=1,\ldots,-\alpha_i \\
& \iff H_iJ_i\not\subset\sigma_i(\sigma_i^{-k}(\frm)) \qquad \forall~k=1,\ldots,-\alpha_i \\
& \iff\sigma_i^{k}(\frm) \text{ is not an $i$-break }\qquad \forall ~k=-1,\ldots,\alpha_i \\ 
& \iff [n_i]^-\prec_i \sigma_i^{k}(\frm) \qquad \forall~k=-1,\ldots,\alpha_i.
\qedhere\end{align*}
\end{proof}
\begin{definition}\label{def.balpha}
Let $\frm$ be a maximal ideal. For every $\balpha \in G_\frm$, since $B_{-\balpha} B_\balpha\not\subset\frm$, we choose $b_\balpha \in B_\balpha$ and $b_\balpha^* \in B_{-\balpha}$ such that $b_\balpha^* b_\balpha \notin \frm$.
As $\frm$ is maximal, then there exists $r \in R$ such that $1-rb_\balpha^*b_\balpha \in \frm$. Set $b_\balpha' = rb_\balpha^* \in B_{-\balpha}$ so that 
\begin{align}\label{eq.brel}
    b_\balpha'b_\balpha \equiv 1 \pmod{\frm}.
\end{align}
\end{definition}

\begin{lemma}\label{lem:basis}
Let $M$ be a simple weight module over $B$, such that $M_\frm\neq 0$, and let $0\neq v_{\frm}\in M_\frm$ be a weight vector, so that $\{v_\frm\}$ is a $\kk$-basis for $M_\frm$. Then the set 
$\{b_\balpha v_\frm~|~\balpha \in G_\frm\}$
is a basis for $M$ over $\kk$.
\end{lemma}
\begin{proof}First of all, notice that $b_\balpha v_\frm\neq 0$ for all $\balpha\in G_\frm$ because, since $b'_{\balpha} b_\balpha\equiv 1 \pmod{\frm}$, we have
$$b'_\balpha(b_\balpha v_{\frm})=(b'_\balpha b_\balpha )v_{\frm}=v_{\frm}\neq 0.$$
Then, for $\balpha\neq \bbeta$ we have that $b_\balpha v_\frm$ and $b_\bbeta v_\frm$ belong to the different weight spaces $M_{\bsigma^{\balpha}(\frm)}$ and $M_{\bsigma^{\bbeta}(\frm)}$ respectively, so they are linearly independent over $R$, hence over $\kk$. Further, by Proposition \ref{prop.dim-weight}, since each weight space of $M$ is one-dimensional, $\{b_\balpha v_\frm\}$ is a $R/\bsigma^{\balpha}(\frm)$-basis for $M_{\bsigma^{\balpha}(\frm)}$ for all $\balpha\in G_\frm$. All that is left to prove, then, is that if $\bgamma\in \ZZ^n\setminus G_\frm$, then $M_{\bsigma^\bgamma(\frm)}=0$. Let $\bgamma=(\gamma_1,\ldots,\gamma_n)\in \ZZ^n\setminus G_\frm$, then by Lemma \ref{lem.Gm} there is an $i$ such that either $\sigma_i^{\gamma_i}(\frm)\succ_i[\frn_i]$ or such that $\sigma_i^{\gamma_i}(\frm)\preceq_i[\frn_i]^-$ (notice that for this $i$, $\gamma_i \neq 0$). We suppose we are in the first of these two cases and therefore that $\gamma_i>0$, the proof for the other case (with $\gamma_i<0$) being analogous. Let $0\leq k<\gamma_i$ be such that $\bsigma_i^{k}(\frm)\in[\frn_i]$, so that in particular $\bsigma_i^{k}(\frm)$ is an $i$-break. We have
\[
M_{\bsigma^{\bgamma}(\frm)}=B_{\bgamma}v_\frm=\left(\prod_{j=1}^n B_{\gamma_j\be_j}\right)v_{\frm}=\left(\prod_{j\neq i} B_{\gamma_j\be_j}\right)B_{\gamma_i\be_i}v_{\frm}=\left(\prod_{j\neq i} B_{\gamma_j\be_j}\right)B_{(\gamma_i-k-1)\be_i}B_{\be_i}B_{k\be_i}v_{\frm}.
\]
Now, $B_{k\be_i}v_{\frm}=M_{\sigma_i^{k}(\frm)}$ and, since $\sigma_i^{k}(\frm)$ is an $i$-break we have by Lemma \ref{lem:cross-break} that 
$$B_{\be_i}B_{k\be_i}v_{\frm}=B_{\be_i} M_{\sigma_i^{k}(\frm)}=0$$
hence $M_{\bsigma^{\bgamma}(\frm)}=0$.
\end{proof}

\begin{figure}
\begin{subfigure}{.45\textwidth}
\centering
    \begin{tikzpicture}[scale = .7]
\draw[red, line width=.5mm] (2, -3)--(2,5);
\draw[red, line width=.5mm] (-2, -3)--(-2,5);
\draw[blue, line width=.5mm] (-5, 4)--(5,4);
\draw[blue, line width=.5mm] (-5, -2)--(5,-2);
\foreach \y in {-1, 0, 1, 2}{\draw[help lines, dashed] (-5, 2*\y)--(5, 2*\y);}
\foreach \x in {-2, -1,0,1,2}{\draw[help lines, dashed] (2*\x, -3)--(2*\x, 5);}
\node[font = \scriptsize, anchor = north west] at (0, 0){\scriptsize $\mathfrak{m}$};
\fill (0, 0) circle (3pt);
\node[font = \scriptsize, anchor = north west] at (2, 0){\scriptsize $\sigma_1(\mathfrak{m})$};
\fill (2, 0) circle (3pt);
\node[font = \scriptsize, anchor = north west] at (2, 2){\scriptsize $\sigma_1\sigma_2(\mathfrak{m})$};
\fill (2, 2) circle (3pt);
\node[font = \scriptsize, anchor = north west] at (2, 4){\scriptsize $\sigma_1\sigma_2^2(\mathfrak{m})$};
\fill (2, 4) circle (3pt);
\node[font = \scriptsize, anchor = north west] at (0, 2){\scriptsize $\sigma_2(\mathfrak{m})$};
\fill (0, 2) circle (3pt);
\node[font = \scriptsize, anchor = north west] at (0, 4){\scriptsize $\sigma_2^2(\mathfrak{m})$};
\fill (0, 4) circle (3pt);

\draw [draw=black, fill = black, opacity = .2] (2.4,4.4) rectangle (-0.4,-0.4);
\end{tikzpicture}
\caption{$G_{\frm}$ or $M\left(\mathcal{O}, ([\sigma_1(\frm)], [\sigma_2^2(\frm)])\right)$}
\end{subfigure}
\begin{subfigure}{.45\textwidth}
\centering
    \begin{tikzpicture}[scale = .7]
\draw[red, line width=.5mm] (2, -3)--(2,5);
\draw[red, line width=.5mm] (-2, -3)--(-2,5);
\draw[blue, line width=.5mm] (-5, 4)--(5,4);
\draw[blue, line width=.5mm] (-5, -2)--(5,-2);
\foreach \y in {-1, 0, 1, 2}{\draw[help lines, dashed] (-5, 2*\y)--(5, 2*\y);}
\foreach \x in {-2, -1,0,1,2}{\draw[help lines, dashed] (2*\x, -3)--(2*\x, 5);}
\node[font = \scriptsize, anchor = north west] at (-2, 2){\scriptsize $\sigma_1^{-1}\sigma_2(\mathfrak{m})$};
\fill (-2, 2) circle (3pt);
\node[font = \scriptsize, anchor = north west] at (-2, 0){\scriptsize $\sigma_1^{-1}(\mathfrak{m})$};
\fill (-2, 0) circle (3pt);
\node[font = \scriptsize, anchor = north west] at (-2, 4){\scriptsize $\sigma_1^{-1}\sigma_2^2(\mathfrak{m})$};
\fill (-2, 4) circle (3pt);
\node[font = \scriptsize, anchor = north west] at (-4, 2){\scriptsize $\sigma_1^{-2}\sigma_2(\mathfrak{m})$};
\fill (-4, 2) circle (3pt);
\node[font = \scriptsize, anchor = north west] at (-4, 0){\scriptsize $\sigma_1^{-2}(\mathfrak{m})$};
\fill (-4, 0) circle (3pt);
\node[font = \scriptsize, anchor = north west] at (-4, 4){\scriptsize $\sigma_1^{-2}\sigma_2^2(\mathfrak{m})$};
\fill (-4, 4) circle (3pt);
\draw [draw=black, fill = black, opacity = .2] (-1.6,-.4) rectangle (-5.6,4.4);
\end{tikzpicture}
\caption{$G_{\sigma_1\inv(\frm)}$ or $M\left(\mathcal{O}, ([\sigma_1\inv(\frm)], [\sigma_2^2(\frm)])\right)$}
\end{subfigure}
\caption{In the terminology of Lemma~\ref{lem.Gm}, the maximal ideals contained in the grey rectangle in the left-hand figure are precisely those $\sigma^\balpha(\frm)$, such that $\balpha\in G_\frm$.
In the right-hand figure, the grey rectangle extends infinitely to the left, representing those $\sigma^\balpha(\sigma_1\inv(\frm))$ such that $\balpha\in G_{\sigma_1\inv(\frm)}$. This picture can also be interpreted, in the terminology of Definition~\ref{def.simplemod}, the left-hand figure depicts $M\left(\mathcal{O}, ([\sigma_1(\frm)], [\sigma_2^2(\frm)])\right)$ while the right-hand figure depicts $M\left(\mathcal{O}, ([\sigma_1\inv(\frm)], [\sigma_2^2(\frm)])\right)$.}
\end{figure}
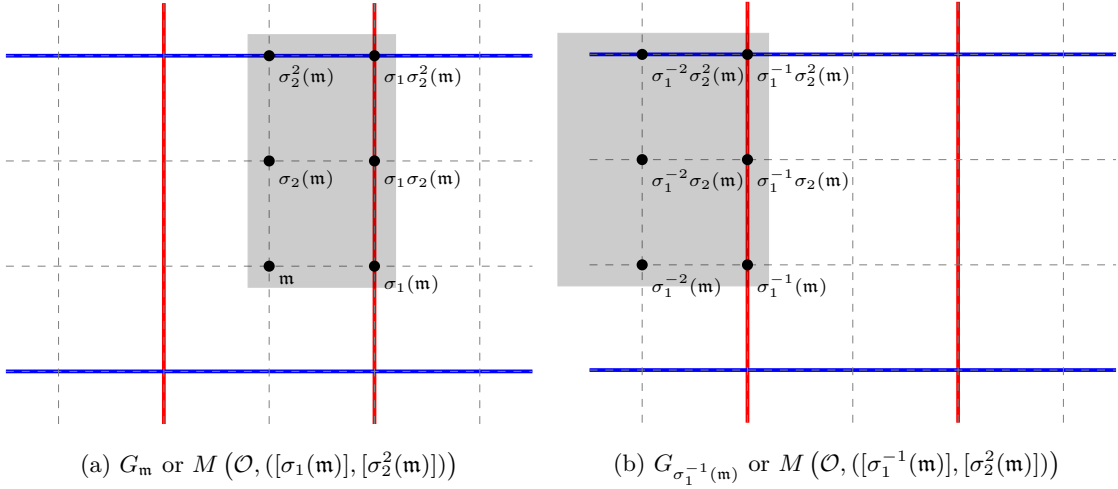

The following lemma is an analogue of \cite[Prop 7.2]{MPT}.

\begin{lemma}\label{lem:exist}Let $\frm$ be a maximal ideal, then there exists a simple weight module $M$ over $B$ such that $M_\frm\neq 0$.
\end{lemma}
\begin{proof}Define the left $B$-module $N:=B\otimes_R R/\frm=\bigoplus_{\balpha\in\ZZ^n}B_\balpha\otimes_R R/\frm$. Let $r\in R$, and let $0\neq x_\balpha t^{\balpha}\in B_\balpha$, then
$$r\cdot x_\balpha t^\balpha\otimes (1_R+\frm)=x_{\balpha} t^{\balpha}\otimes \bsigma^{-\balpha}(r)(1_R+\frm)$$
which equals zero if and only if $\bsigma^{-\balpha}(r)\in\frm$ if and only if $r\in\sigma^{\balpha}(\frm)$. This proves that $N$ is a weight module, with $N_{\bsigma^{\balpha}(\frm)}=B_\balpha\otimes_R R/\frm$. Further, we have that $N_\frm=B_{\bzero}\otimes_R R/\frm=R\otimes_R R/\frm\simeq R/\frm\neq 0$ is a nonzero weight space and $N=BN_\frm$. Now define a weight submodule 
$$N':=\left\{\sum S~|~S\subset N \text{ is a weight submodule and }S\cap N_\frm=0\right\}.$$ We claim that $M=N/N'$ is the required simple weight module for $B$ with $M_\frm\neq 0$. Since $N'_{\frm}=0$, it is clear that $M_\frm\neq 0$. Suppose we have a non trivial weight submodule $0\neq K\subseteq N/N'$. This is the same as a weight submodule $N'\subsetneq K\subseteq N$, but since $N'\subsetneq K$ we have $K_\frm\neq 0$, which implies that $K_\frm=N_\frm=B_{\bzero}\otimes_R R/\frm$, hence $K=BK\supset BK_\frm=BN_\frm=N$, which implies that $K=N$. This proves that $M=N/N'$ is indeed a simple weight module.
\end{proof}

\begin{theorem}\label{thm.weight}
Let $\cO\in\Maxspec(R)/\ZZ^n$ be a torsion-free orbit, let $\beta_i \subseteq \cO$ be the set of $i$-breaks for each $i$, and keep the notation above. The isomorphism classes of simple modules in $(B,R)\wmodO$ are in bijection with the set $\prod_{i=1}^n\overline{\beta}'_i$. If $M\in (B,R)\wmodO$ is simple, it corresponds to the unique $n$-tuple $([\frn_1],\ldots,[\frn_n])\in \prod_{i=1}^n\overline{\beta}'_i$ such that $\supp(M)=\{\frm\in\cO~|~[\frn_i]^-\prec_i\frm\preceq_i[\frn_i]\text{ for all }i=1,\ldots,n\}$.
\end{theorem}
\begin{proof}
Consider the map from isomorphism classes of simple weight modules to $n$-tuples of breaks as in the statement. This map is well defined because, by Lemma \ref{lem:basis}, if $M$ is a simple weight module with $M_{\frm}\neq 0\neq M_{\frm'}$, then $\frm'=\bsigma^{\bbeta}(\frm)$ for some $\bbeta\in G_{\frm}$, hence, by Lemma \ref{lem.Gm}, $\frm'$ satisfies the same inequalities with respect to the breaks $[\frn_i]$ as $\frm$ does. Also, this map is surjective by Lemma \ref{lem:exist}, so we only need to show that it is injective. Suppose that we have two simple weight modules $M$ and $M'$ that map to the same $n$-tuple of breaks, then they have the same support. Let $\frm$ be a maximal ideal such that $M_\frm\neq 0\neq M'_\frm$ and we can choose $b_\balpha$, $\balpha\in G_\frm$, and $0\neq v_\frm\in M_\frm$, $0\neq v'_\frm\in M'_\frm$ such that $\{b_\balpha v_\frm~|~\balpha\in G_\frm\}$ is a $\kk$-basis for $M$ and $\{b_\balpha v'_\frm~|~\balpha\in G_\frm\}$ is a $\kk$-basis for $M'$. Then, the map defined by $b_\balpha v_\frm\mapsto b_\balpha v'_\frm$ for all $\balpha\in G_\frm$ is the required isomorphism of weight modules.
\end{proof}

\begin{definition}
\label{def.simplemod}
For each $\underline{[\frn]}=([\frn_1],\ldots,[\frn_n])\in \prod_{i=1}^n\overline{\beta}'_i$, denote by $M(\cO,\underline{[\frn]})$ the corresponding simple weight module.
\end{definition}

We can give a somewhat explicit description of the structure constants for the action of $B$ on $M(\cO,\underline{[\frn]})$, depending on the choice of $b_\balpha,b'_\balpha$, compare the following result to \cite[Thm 5.4]{hart1} for TGWAs.

\begin{proposition}
    Let $\underline{[\frn]}=([\frn_1],\ldots,[\frn_n])\in \prod_{i=1}^n\overline{\beta}'_i$. 
    Choose $\frm$ to be a maximal ideal such that $[\frn_i]^-\prec_i\frm\preceq_i[\frn_i]$ for all $i$, and choose $b_\balpha,b'_\balpha$, $\balpha\in G_\frm$ as in Definition \ref{def.balpha}. We have an isomorphism of $R$-modules
\[M(\cO,\underline{[\frn]})\simeq\bigoplus_{\balpha\in G_\frm}\left(R/\bsigma^{\balpha}(\frm)\right)( v_{\frm,\balpha})\qquad \]
 with $B$-action defined by 
\[
j_it_i  v_{\frm,\balpha} = \begin{cases}
    j_i x_{\balpha,i} v_{\frm,\balpha+\be_i} & \balpha+\be_i \in G_\frm \\
     0 & \text{otherwise,}
 \end{cases} \qquad
\sigma_i\inv(h_i)t_i\inv  v_{\frm,\balpha} = \begin{cases}
    \sigma_i\inv(h_i)y_{\balpha,i} v_{\frm,\balpha-\be_i} & \balpha-\be_i \in G_\frm \\
     0 & \text{otherwise,}
 \end{cases}
\]
for all $j_i \in J_i$ and $h_i \in H_i$
where $x_{\balpha,i} = t_i b_\balpha b_{\balpha+\be_i}'$
and
$y_{\balpha,i} = t_i\inv b_\balpha b_{\balpha-\be_i}'$ are elements of $B_\bzero=R$.
\end{proposition}
\begin{proof}
Let $\frm$ be a maximal ideal as in the statement, and choose $b_\balpha, b'_\balpha$ for $\balpha\in G_\frm$. Let $0\neq v_\frm$ be any vector in $M(\cO,\underline{[\frn]})_\frm$, and define $v_{\frm,\balpha}:=b_\balpha v_\frm$ for all $\balpha\in G_\frm$. By Lemma \ref{lem:basis}, $\{v_{\frm,\balpha}~|~\balpha\in G_\frm \}$ is a $\kk$-basis of $M(\cO,\underline{[\frn]})$, and since $b_\balpha\in B_\balpha$, it is clear that $v_{\frm,\balpha}\in M(\cO,\underline{[\frn]})_{\sigma^{\balpha}(\frm)}$, hence the isomorphism as $R$-modules.
Now, if $\balpha\in G_\frm$, $\balpha+\be_i\not\in G_\frm$, then $\bsigma^{\balpha}(\frm)$ is an $i$-break by Lemma \ref{lem.Gm}, hence by Lemma \ref{lem:cross-break} we have indeed that $j_it_i  v_{\frm,\balpha}=0$. If $\balpha+\be_i\in G_\frm$, then, since $b_{\balpha+\be_i}'b_{\balpha+\be_i}\equiv 1 \pmod \frm$,
$$ j_it_i  v_{\frm,\balpha}=j_it_i b_{\balpha} v_{\frm}=j_it_ib_{\balpha}b_{\balpha+\be_i}'b_{\balpha+\be_i}v_{\frm}=j_i x_{\balpha,i} v_{\frm,\balpha+\be_i}.$$
If $\balpha\in G_\frm$, $\balpha-\be_i\not\in G_\frm$, then $\sigma_i\inv\sigma^{\balpha}(\frm)$ is an $i$-break by Lemma \ref{lem.Gm}, hence by Lemma \ref{lem:cross-break} we have indeed that $\sigma_i\inv(h_i)t_i\inv  v_{\frm,\balpha}=0$. If $\balpha-\be_i\in G_\frm$, then, since $b_{\balpha-\be_i}'b_{\balpha-\be_i}\equiv 1 \pmod \frm$,
\[\sigma_i\inv(h_i)t_i\inv  v_{\frm,\balpha}=\sigma_i\inv(h_i)t_i\inv b_{\balpha} v_{\frm}=\sigma_i\inv(h_i)t_i\inv b_{\balpha}b_{\balpha-\be_i}'b_{\balpha-\be_i}v_{\frm}=  \sigma_i\inv(h_i)y_{\balpha,i} v_{\frm,\balpha-\be_i}.\qedhere\]
\end{proof}

Let $M\in (B,R)\wmodO$ with $\cO=\ZZ^n\cdot \frm$ a torsion-free orbit. As in the rank one case \cite{GRW1}, $M$ may be given the structure of a $\ZZ^n$-graded $B$-module by defining $M_\balpha=M_{\sigma^\balpha(\frm)}$ for all $\balpha\in\ZZ^n$. The grading, however, is not canonical. Choosing different points in the orbit to correspond to $\bzero$ recovers shifts of $M$ in the graded module category. However, as shown in \cite{GRW1}, a (non-simple) graded module need not necessarily be an $R$-weight module. 

\begin{proposition}
Let $M$ be a simple $\ZZ^n$-graded $B$-module. Then $M$ is an $R$-weight module.
\end{proposition}
\begin{proof}
Write $M=\bigoplus_{\balpha\in\ZZ^n}M_\balpha$ so that $B_\bbeta M_\balpha\subseteq M_{\balpha+\bbeta}$ for all $\balpha,\bbeta \in \ZZ^n$. If $\balpha\in\ZZ^n$ such that $M_\balpha\neq 0$, then $M_\balpha$ is an $R$-submodule since $RM_\balpha=B_\bzero M_\balpha\subseteq M_\balpha$. We claim $M_\balpha$ is a simple $R$-module. 

Suppose $N_\balpha$ is a proper, nontrivial submodule of $M_\balpha$. Note that $(BN_\balpha)_\balpha=N_\balpha$ and $0\neq BN_\balpha\subsetneq M$, so $BN_\balpha$ is a nontrivial $B$-submodule of $M$, contradicting the simplicity of $M$. This proves the claim, so $M_\balpha\simeq R/\frm$ as $R$-modules for some $\frm\in\Maxspec(R)$. 

By simplicity, $M=BM_\balpha$. Hence, for all $\bbeta \in \ZZ^n$, $M_\bbeta = B_{\bbeta-\balpha}M_\balpha$. 
If $r\in \bsigma^{\bbeta-\balpha}(\frm)$, then $\bsigma^{\balpha-\bbeta}(r) \in \frm$, so
\[rM_\bbeta=rB_{\bbeta-\balpha}M_\balpha=B_{\bbeta-\balpha}\sigma^{\balpha-\bbeta}(r)M_\balpha=0.\]
Thus, $M_\bbeta\subseteq M_{\bsigma^{\bbeta-\balpha}(\frm)}$ and so
\[M=\bigoplus_{\bbeta\in \ZZ^n}M_\bbeta\subseteq \bigoplus_{\bbeta\in\ZZ^n}M_{\bsigma^{\bbeta-\balpha}(\frm)}\subseteq M.\]
It follows that $M=\bigoplus_{\bbeta\in\ZZ^n}M_{\bsigma^{\bbeta-\balpha}(\frm)}$ is a weight module.
\end{proof}

\section{Twisted tensor products of BR algebras}
\label{sec.twtensor}

The goal of this section is to show that the class of BR algebras is closed under taking certain twisted tensor products. Restricting to ordinary tensor products then gives new examples of $\ZZ^n$-graded simple rings.

\begin{definition}[\cite{cap}]\label{def.twtensor}
Let $B$ and $D$ be algebras with multiplication maps $\mu_B$ and $\mu_D$, respectively. A \emph{twisted tensor product} of $B$ and $D$ is an algebra $W$ along with injections $\iota_B:B \to W$ and $\iota_D:D \to W$ such that there is a linear isomorphism $B \tensor D \to W$ given by $b \tensor d \mapsto \iota_B(b) \cdot \iota_D(d)$.
A \emph{twisting map} for the pair $(B,D)$ is a $\kk$-linear map $\tau:D \tensor B \to B \tensor D$ such that $\tau(d \tensor 1)=1 \tensor d$ and $\tau(1 \tensor b) = b \tensor 1$. 
\end{definition}

Given a twisting map for algebras $B$ and $D$, one can define a twisted tensor product as follows. Let $W = B \otimes D$ (as a $\kk$-vector space) and define a map $\mu_W: W \otimes W \to W$ by
\[ 
\mu_W : = (\mu_B \tensor \mu_D) \circ (1_B \tensor \tau \tensor 1_D).
\] 
By \cite[Proposition 2.3]{cap}, $\mu_W$ defines an associative multiplication operation on $C$ if and only if
\[ \tau \circ (\mu_D \tensor \mu_B) 
    = \mu_W \circ (\tau \tensor \tau) \circ (1_D \tensor \tau \tensor 1_B)
    = (\mu_B \tensor \mu_D) \circ (1_B \tensor \tau \tensor 1_D)\circ (\tau \tensor \tau) \circ (1_D \tensor \tau \tensor 1_B).
\]
If this holds, then $W$ is a twisted tensor product of $B$ and $D$ with canonical injections $\iota_B: B \to B \otimes D$ and $\iota_D: D \to B \otimes D$. We write $W = B \otimes_{\tau} D$ for this twisted tensor product.

We will assume the following hypothesis throughout this section.

\begin{hypothesis}\label{hyp.twtensor}
Let $R$ and $S$ be algebras. Let $A = R_\bp[\bu^{\pm 1};\bsigma]$ be a skew Laurent extension over $R$ and let $B = R_{\bp}(\bu,\bsigma,\bH,\bJ) \subseteq A$ a BR algebra of rank $m$ with canonical ideals $I^{(\balpha)}$ for $\balpha \in \ZZ^m$.
Let $C = S_{\bq}[\bv^{\pm 1};\bphi]$ be a skew Laurent extension over $S$ and let $D = S_{\bq}(\bv,\bphi,\bK,\bL) \subset C$ be a BR algebra of rank $n$ with canonical ideals $\cI^{(\bbeta)}$ for $\bbeta \in \ZZ^n$. 
Set $\bsigma = (\sigma_1, \dots, \sigma_m)$ and $\bphi = (\varphi_1, \dots ,\varphi_n)$. For each $i \in [m]$, suppose that $\sigma_i$ lifts to an automorphism $\widetilde{\sigma_i}$ of $R \otimes S$ such that $\widetilde{\sigma_i}(1 \otimes \cI^{(\bbeta)}) \subseteq 1 \otimes \cI^{(\bbeta)}$ for all $\bbeta \in \ZZ^n$.
Similarly, suppose that for each $k \in [n]$, $\varphi_k$ lifts to an automorphism $\widetilde{\varphi_k}$ of $R \otimes S$ such that $\widetilde{\varphi_k}(I^{(\balpha)} \otimes 1) \subseteq I^{(\balpha)} \otimes 1$ for all $\balpha \in \ZZ^m$.
Assume further that all of the $\widetilde{\sigma_i}$ and $\widetilde{\varphi_k}$ commute with one another.
\end{hypothesis}

\begin{remark}\label{rmk.autrestr}
Assume Hypothesis~\ref{hyp.twtensor}. 
\begin{enumerate}
    \item One may always lift the $\sigma_i$ and $\varphi_k$ trivially to commuting automorphisms $\widetilde{\sigma_i} = \sigma_i \otimes \id_S$ and $\widetilde{\varphi_k} = \id_R \otimes \varphi_k$.
    \item By taking $\bbeta = \bzero$, it is clear that the restriction of $\widetilde{\sigma_i}$ to $1 \otimes S$ yields an automorphism of $S$, which we also denote by $\widetilde{\sigma_i}$ by an abuse of notation. Similarly, the restriction of $\widetilde{\varphi_k}$ to $R \otimes 1$ yields an automorphism of $R$.
\end{enumerate}
\end{remark}

For $i \in [m]$ and $k \in [n]$, choose $d_{ik} \in \kk^\times$. For $\bbeta \in \ZZ^m$ and $\balpha \in  \ZZ^n$, define 
\begin{align}\label{eq.dba}
d_{\bbeta,\balpha} = \prod_{\substack{\beta_i \in \bbeta \\ \alpha_k \in \balpha}} {d_{ik}^{\beta_i\alpha_k}}.
\end{align}
We define a $\kk$-linear map $\tau: C \tensor A \to A \tensor C$ as follows. If $b_{\bbeta} u^{\bbeta} \in A_{\bbeta}$ and $a_{\balpha} v^{\balpha} \in C_{\balpha}$, then
\begin{align}\label{eq.SLtau}
\tau( a_\balpha \bv^\balpha \tensor b_\bbeta \bu^\bbeta) = d_{\bbeta,\balpha} \widetilde{\bphi}^\balpha(b_\bbeta) \bu^\bbeta \tensor \widetilde{\bsigma}^{-\bbeta}(a_\balpha) \bv^\balpha,
\end{align}
where $\widetilde{\bphi}^\balpha = \widetilde\varphi_1^{\alpha_1} \cdots \widetilde\varphi_n^{\alpha_n}$ and $\widetilde{\bsigma}^{-\bbeta} = \widetilde\sigma_1^{-\beta_1} \widetilde\cdots \sigma_m^{-\beta_m}$. 
The map $\tau$ restricts to a well-defined map $D \tensor B \to B \tensor D$ by our assumption on the automorphisms $\widetilde{\sigma_i}$ and $\widetilde{\varphi_k}$ preserving the canonical ideals.

\begin{lemma}
The map $\tau$ defined in \eqref{eq.SLtau} satisfies the associativity conditions.
\end{lemma}
\begin{proof}
Let $a_\balpha,b_\bbeta,c_\bgamma,d_\bdelta$ belong to the correct rings. Then
\begin{align*}
(\mu_B &\tensor \mu_{\cB})(1 \tensor \tau \tensor 1)(\tau \tensor \tau)(1 \tensor \tau \tensor 1)(a_\balpha \bv^\balpha \tensor b_\bbeta \bv^\bbeta \tensor c_\bgamma \bu^\bgamma \tensor d_\bdelta \bu^\bdelta) \\
	&= d_{\bbeta,\bgamma}(\mu_A \tensor \mu_B)(1 \tensor \tau \tensor 1)(\tau \tensor \tau)(a_\balpha \bv^\balpha \tensor \widetilde{\bphi}^\bbeta(c_\bgamma) \bu^\bgamma \tensor \widetilde{\bsigma}^{-\bgamma}(b_\bbeta) \bv^\bbeta  \tensor d_\bdelta \bu^\bdelta) \\
	&= d_{\bbeta,\bgamma}d_{\balpha,\bgamma}d_{\bbeta,\bdelta}(\mu_A \tensor \mu_B)(1 \tensor \tau \tensor 1)(\widetilde{\bphi}^{\balpha+\bbeta}(c_\bgamma) \bu^\bgamma \tensor \widetilde{\bsigma}^{-\bgamma}(a_\balpha) \bv^\balpha \tensor \widetilde{\bphi}^{\bbeta}(d_\bdelta) \bu^\bdelta   
    \tensor  \widetilde{\bsigma}^{-(\bgamma+\bdelta)}(b_\bbeta) \bv^\bbeta) \\
	&= d_{\balpha+\bbeta,\bgamma+\bdelta}(\mu_B \tensor \mu_{\cB})( 
    \widetilde{\bphi}^{\balpha+\bbeta}(c_\bgamma) \bu^\bgamma \tensor \widetilde{\bphi}^{\balpha+\bbeta}(d_\bdelta) \bu^\bdelta 
    \tensor \widetilde{\bsigma}^{-(\bgamma+\bdelta)}(a_\balpha) \bv^\balpha \tensor   
    \widetilde{\bsigma}^{-(\bgamma+\bdelta)}(b_\bbeta) \bv^\bbeta) \\
	&= d_{\balpha+\bbeta,\bgamma+\bdelta} \left( \widetilde{\bphi}^{\balpha+\bbeta}(c_\bgamma\widetilde{\bsigma}^\bgamma(d_\bdelta)) \bu^{\bgamma+\bdelta} \tensor 
	\widetilde{\bsigma}^{-(\bgamma+\bdelta)}(a_\balpha\widetilde{\bphi}^\balpha(b_\bbeta)) \bv^{\balpha+\bbeta} \right) \\
    &= \tau\left( a_\balpha \widetilde{\bphi}^\balpha(b_\bbeta) \bv^{\balpha+\bbeta} \tensor c_\bgamma\widetilde{\bsigma}^\bgamma(d_\bdelta) \bu^{\bgamma+\bdelta} \right) \\
    &= \tau \circ (\mu_{\cB} \tensor \mu_B)(a_\balpha \bv^\balpha \tensor b_\bbeta \bv^\bbeta \tensor c_\bgamma \bu^\bgamma \tensor d_\bdelta \bu^\bdelta).\qedhere
\end{align*}
\end{proof}

We now construct a BR algebra $W$ and show that $B \tensor_\tau D \iso W$.
Let $\bd=(d_{ij})$ and set
\[ \br = \begin{pmatrix}\bp & \bd \\ \bd^T & \bq\end{pmatrix}.\]
Set $\bpi=(\pi_1,\hdots,\pi_{m+n})$ where $\pi_i$ is the extension of $\sigma_i$ to $R \tensor S$ for $i\in [m]$ and $\pi_{m+j}$ is the extension of $\varphi_j$ for $j \in [n]$. 
The $\pi_i$ then commute by hypothesis. Let $(R \tensor S)_\br[\bt^{\pm 1},\bpi]$ be the rank $m+n$ skew Laurent extension corresponding to this data.

\begin{lemma}\label{lem.twtensor}
Keep the above notation and let $\tau$ be as in \eqref{eq.SLtau}. Then $A\tensor_\tau C \iso (R \tensor S)_\br[\bt^{\pm 1},\bpi]$.
\end{lemma}
\begin{proof}
Define a map $\phi:(R \tensor S)_\br[\bt^{\pm 1},\bpi] \to A\tensor_\tau  C$ that is the identity on $R \tensor S$ and
\[
\phi(t_i) = \begin{cases}
    u_i \tensor 1 & \text{if $i\leq m$} \\
    1 \tensor  v_{i-m} & \text{if $m < i \leq m+n$.}
\end{cases}
\]
If suffices to check $\phi$ on the defining relations of $(R \tensor S)_\br[\bt^{\pm 1},\bpi]$. For $i \leq m$ we have
\begin{align*}
    \phi(t_i)\phi(r \tensor r') - \phi(\pi_i(r \tensor r'))\phi(t_i).
\end{align*}
A similar result holds for $i > m$.
If $i,k \leq m$, then it is clear that $\phi(t_k)\phi(t_i) = r_{ik}\phi(t_i)\phi(t_k)$,
as is the case when $i,k>m$. Finally, for $i \leq m$ and $k > m$, we have
\[ \phi(t_k)\phi(t_i) 
    = (1 \tensor v_{k-m}) \tensor (u_i \tensor 1)
    = d_{\be_i,\be_{k-m}} t_i \tensor t_{k-m}
    = r_{ik} (u_i \tensor 1) \tensor (1 \tensor v_{k-m})
    = r_{ik} \phi(t_i)\phi(t_k).
\]
The result follows.
\end{proof}

Now define tuples of ideals $\bM$ and $\bN$ of $R \tensor S$ as follows: 
\begin{align*}
\bM &= (M_1,\ldots, M_{m+n}) = (H_1 \tensor S,\hdots,H_m\tensor S,R \tensor K_1,\hdots,R \tensor K_n) \\
\bN &= (N_1,\ldots, N_{m+n}) =  (J_1 \tensor S,\hdots,J_m\tensor S,R \tensor L_1,\hdots,R \tensor L_n).
\end{align*}
By hypothesis, the $\pi_i$ commute. Moreover, $\pi_i(M_k)=M_k$ and $\pi_i(N_k)=N_k$ for $i \neq k$. Define the BR algebra $W = (R \tensor S)_\br(\bs,\bpi,\bM,\bN)$, which is a subalgebra of $(R \tensor S)_\br[\bs^{\pm 1},\bpi]$.

\begin{theorem}\label{thm.twtensor}
Assume Hypothesis \ref{hyp.twtensor} and let $\tau$ be as in \eqref{eq.SLtau}. Then $B \tensor_\tau D \iso W$.
\end{theorem}
\begin{proof}
It is clear that $B \tensor_\tau D$ is a subalgebra of $A \tensor_\tau C$. Let $\phi$ be the map given in Lemma \ref{lem.twtensor} and consider the restriction $\Phi=\phi\restrict{W}$. By Remark \ref{rmk.autrestr}(2), $\Phi$ is a well-defined map $W \to B \tensor_\tau D$. That $\Phi$ is injective follows from the injectivity of $\phi$. 

Let $j_i \in J_i$ so that $j_iu_i$ is a generator for $B$, $1 \leq i \leq m$. Then $j_i \tensor 1 \in N_i$ so $\Phi((j_i \tensor 1)t_i) = j_iu_i \tensor 1$. Similarly we obtain $\sigma_i\inv(h_i)u_i\inv \tensor 1 \in \im\Phi$ for each $i$ and all $h_i \in H_i$, as well as $1 \tensor l_iv_i, 1 \tensor \varphi_i\inv(k_i)v_i\inv \in \im\Phi$ for all $i$, $1 \leq i \leq n$ and all $l_i \in L_i$, $k_i \in K_i$. It follows from Lemma \ref{lem.BRgen} that the subalgebras $B \tensor 1$ and $1 \tensor D$ are in the image of $\Phi$.

Now let $\sum b_\bbeta \tensor d_\bdelta \in B \tensor_\tau D$ where $\bbeta$ ranges over $\ZZ^m$ and $\bdelta$ ranges over $\ZZ^n$. For each $\bbeta$, $b_\bbeta \in B_\bbeta$ and so $b_\bbeta \tensor 1 \in B \tensor_\tau D$. Similarly, $1 \tensor d_\bdelta \in B \tensor_\tau D$. 
It follows from \eqref{eq.SLtau} that $b_\bbeta \tensor d_\bdelta = (b_\bbeta \tensor 1)(1 \tensor d_\bdelta) \in B \tensor_\tau D$. Thus, $\Phi$ is surjective.
\end{proof}

\begin{corollary}\label{cor.tensor}
Assume Hypothesis \ref{hyp.twtensor}. Then the (untwisted) tensor product $B \tensor D$ is a BR algebra of rank $m+n$.
\end{corollary}

We now construct new examples of $\ZZ^n$-graded simple rings, for all $n \geq 1$. 

\begin{proposition}
[{\cite[Proposition 2.18 (3)]{BR}}, {\cite[Proposition 2.3]{GRW1}}]
\label{prop.simple}
Let $B=R(t,\sigma,H,J)$ be a BR algebra of rank one and set $\css(B)=\css_1(B)$. Then $B$ is simple if and only if 
\begin{itemize}
    \item $R$ has no proper, nontrivial $\sigma$-invariant ideals ($R$ is $\sigma$-simple), and
    \item $\sigma^i(\css(B))\cap\css(B) = \emptyset$ for all $i \neq 0$ ($B$ is $\sigma$-lonely).
\end{itemize}
\end{proposition}

\begin{example}
Set $R=\kk[u^{\pm 1},v^{\pm 1}]$ and define $\sigma \in \Aut(R)$ by $\sigma(u)=pu$ and $\sigma(v)=qv$ with $p,q \in \kk^\times$ such that $G=\grp{p,q}$ is a torsion-free abelian group of rank two. Then $R$ is $\sigma$-simple \cite[Theorem 2.4]{sham}. Set $H=R$ and $J=(u+v,(u+1)^2)$. It is clear that $\css(B)=\{ (u+1,v-1) \}$. Then $\sigma^i(\css(B)) = \{ (u+p^{-i},v-q^{-i}) \}$, so $B$ is $\sigma$-lonely. It follows that $B=R(t,\sigma,H,J)$ is simple.

Since we have assumed that $G$ is torsion-free, then $\sigma$ has infinite order. Hence, by \cite[Lemma 2.1]{GRW1}, $\cnt(B)=\cnt(R)^{\grp{\sigma}}$. An easy check shows that $\cnt(R)^{\grp{\sigma}}=\kk$. It is well-known that the tensor product of two simple algebras, one of which is central simple, is again simple. Thus, $A=\bigotimes_{k=1}^n B$ is a simple BR algebra of rank $n$ by Corollary \ref{cor.tensor}.
\end{example}

\section{A simplicity criterion}
\label{sec.simplicity}

In what follows, we assume that $R$ is a commutative noetherian domain and $B=R_{\bp}(\bt,\bsigma,\bH,\bJ)$ is a BR algebra of rank $n$. The enterprising reader could attempt to remove the domain hypothesis, which simplifies many of our arguments.

Let $\Gamma$ be the subgroup of $\Aut(R)$ generated by $\{\sigma_i: i \in [n]\}$. We say $R$ is \emph{$\Gamma$-simple} if for any ideal $I$, $\sigma(I)=I$ for all $\sigma \in \Gamma$ implies $I=0$ or $I=R$. 

In this section we give a simplicity criterion for BR algebras. We are motivated by the following result of Hartwig and \"{O}inert, and we will follow their methods. This criterion reduces to the one given by Bavula for GWAs \cite[Theorem 4.5]{B3}.

\begin{theorem}[Hartwig, \"{O}inert {\cite[Theorem 7.20]{HO}}]
Let $A=A_\mu(R,\sigma,t)$ be a regularly graded TGWA which is $R$-finitistic of type $(A_1)^n$. Then $A$ is simple if and only if the following hold:
\begin{enumerate}
    \item $R$ is $\Gamma$-simple;
    \item $Z(A) \subset R$;
    \item $R\sigma_i^d(t_i)+Rt_i = R$ for all $i \in 1,\hdots,n$ and all $d \in \ZZ_{>0}$.
\end{enumerate}
\end{theorem}

For $\balpha,\bbeta \in \ZZ^n$, let
\begin{align}\label{eq.pba}
p_{\bbeta,\balpha} = \prod_{\substack{\beta_i \in \bbeta \\ \alpha_k \in \balpha}} {p_{ik}^{\beta_i\alpha_k}}.
\end{align}
Observe that $p_{\be_i,\be_k} = p_{ik}$ and $p_{\bbeta,\balpha}\inv = p_{\balpha,\bbeta}$. In general, we have $t^\balpha t^\bbeta = p_{\bbeta,\balpha} t^\bbeta t^\balpha$.
We write $\balpha > 0$ to indicate $\alpha_i \geq 0$ for all $i$ and $\balpha \neq \bzero$.

For $b = \sum_{\bbeta \in \ZZ^n} b_\bbeta t^{\bbeta} \in B$, we define $\supp(b)=\{ \bbeta : b_\bbeta \neq 0\}$. We denote the cardinality of $\supp(b)$ by $|\supp(b)|$. By definition, $|\supp(b)|<\infty$.

\begin{lemma}\label{lem.multset}
The multiplicative set 
\[ \cX=\left( \{a t^\balpha~|~a\in I^{(\balpha)}, \balpha > 0\}\cup\{1\} \right)\setminus\{0\}.\]
is a left and right Ore set of $B$.
\end{lemma}

\begin{proof}
We will prove that $\cX$ is a right Ore set. The proof that it is a left Ore set is similar. Suppose $x=at^{\balpha} \in I^{(\balpha)}t^\balpha$ for some $\balpha > 0$ and let $g \in B$. We claim that $xB \cap g\cX \neq \emptyset$.

Let $w = \prod_{\bbeta \in \supp(g)} \sigma^{-\bbeta}(a)$. Set $\bgamma=(\gamma_i)$ where for each $i \in [n]$,
\[ \gamma_i = \min\left( \{ \bbeta_i : \bbeta \in \supp(g)\} \cup \{0\} \right).\]
Choose $ct^{\balpha-\bgamma} \in \cX$ so that $wct^{\balpha-\bgamma} \in \cX$.

Let $b_{\bdelta} t^{\bdelta}$ be a nonzero summand in $g$. Set $\widehat{w} = (\sigma^{-\bdelta}(a))\inv w$. Then
\begin{align*}
(b_{\bdelta} t^{\bdelta})(wct^{\balpha-\bgamma})
	&= (b_{\bdelta} t^{\bdelta})(\sigma^{-\bdelta}(a) \widehat{w} ct^{\balpha-\bgamma}) \\
	&= a b_{\bdelta} \sigma^{\bdelta}(\widehat{w}c) t^{\bdelta+\balpha-\bgamma} \\
	&= (a t^{\balpha}) \sigma^{-\balpha}(b_{\bdelta}) \sigma^{\bdelta-\balpha}(\widehat{w}c) t^{\bdelta-\bgamma}.
\end{align*}
Since $c \in I^{(\balpha-\bgamma)}$, then $\sigma^{\bdelta-\balpha}(c) \in I^{(\bdelta-\bgamma)}$ so that $\sigma^{-\balpha}(b_{\bdelta}) \sigma^{\bdelta-\balpha}(\widehat{w}c) t^{\bdelta-\bgamma} \in B$.
Thus, $g(wct^{\balpha-\bgamma}) \in (at^{\balpha})B$, as desired.
\end{proof}

\begin{theorem}\label{thm.loc}
Let $Q=Q(R)$ be the fraction field of $R$. Then $B\cX\inv \iso Q_\bp[\bt^{\pm 1}; \bsigma]$.
\end{theorem}
\begin{proof}
Let $r \in R$ and choose $a \in J_1$ so that $ra \in J_1$. Then $rat_1 \in \cX$ so we have
\[ (at_1)(rat_1)\inv = (at_1)(t_1\inv a\inv r\inv) = r\inv \in B\cX\inv.\]
This implies, in particular, that $a\inv \in B\cX\inv$, so that $a\inv (at_1) \in b\cX\inv$. Similarly, $t_i\inv \in B\cX\inv$ for all $i$. Thus, $Q_\bp[\bt^{\pm 1}; \bsigma] \subset B\cX\inv$.

On the other hand, it is clear that we have the inclusion $B \hookrightarrow Q_\bp[\bt^{\pm 1}; \bsigma]$. Any element of strictly positive degree in $B$ can be written as $at^\balpha$ for some $a \in I^{(\balpha)}$ and $\balpha > 0$. Then $(t^\balpha)\inv a\inv$ is its inverse, and this element is contained in $Q_\bp[\bt^{\pm 1}; \bsigma]$. Hence, each positive degree element of $B$ is a unit in $Q_\bp[\bt^{\pm 1}; \bsigma]$ and so by the universal property of localization, we have the inclusion $B\cX\inv \hookrightarrow Q_\bp[\bt^{\pm 1}; \bsigma]$.
\end{proof}

Henceforth, we set $C=B\cX\inv$. By Theorem \ref{thm.loc}, $C$ is the graded quotient ring of $B$ and is, trivially, a TGWA of type $(A_1)^n$.
Consequently, we may apply results from \cite[Section 7]{HO} directly to this ring. However, as some of our notation is diferent and because we will occasionally need to refer to elements of the proofs, we maintain these arguments for the benefit of the reader.

Let $K$ be an ideal in $B$, then by \cite[Lemma 7.9]{HO},
\begin{align}\label{eq.ideal}
CKC = \{x\inv ay\inv : x,y \in \cX, a \in K\}.
\end{align}

\begin{lemma}\label{lem.cnt1}
\begin{enumerate}
\item For any nonzero $c \in C$, there exists a nonzero element $r \in R$ such that $rc \in B$.
\item If $BxB=B$ for all $x \in \cX$, then $\cnt(B)=\cnt(C)$.
\end{enumerate}
\end{lemma}
\begin{proof}
(1) Let $c \in C$ be nonzero. Then $c=(j_{i_1}t_{i_1})^{-1}\cdots(j_{i_k}t_{i_k})^{-1}b'$ for some $b' \in B$. Note that the $t_{i_k}$ need not be distinct. Since $(j_{i_1}t_{i_1})\inv=t_{i_1}^{-1} j_{i_1}\inv$, then 
\[ \sigma_{i_1}\inv(j_{i_1}\inv)(t_{i_1}^{-1} j_{i_1}\inv) = t_{i_1}^{-1} j_{i_1}j_{i_1}\inv = t_{i_1}^{-1}.\]
By repeated applications of this, we may assume, up to multiplication by an element in $R$, that $c = t_{i_1}\inv \cdots t_{i_k}\inv b$ for some $b \in B$. Let $v \in H_{i_k}$ and set $v' = (\sigma_{i_1}\inv \circ \cdots \sigma_{i_{k-1}}\inv \circ \sigma_{i_k}\inv)(v) \in R$. Then
\[ v' c =  t_{i_1}\inv \cdots t_{i_{k-1}}\inv \sigma_{i_k}\inv(v)t_{i_k}\inv b.\]
But $\sigma_{i_k}\inv(v)t_{i_k}\inv \in B$. The result now follows by induction. 

(2) That $\cnt(B) \subset \cnt(C)$ is clear. Let $c \in \cnt(C)$. Write $c=x\inv b$ for some $x \in \cX$, $b \in B$. Since $1=\sum u_i x v_i$ for some $u_i,v_i \in B$, then
\[ c = (1)c = \sum u_i x v_i c = \sum u_i (xc) v_i = \sum u_i b v_i \in B.\]
Thus, $c \in B$, so $c \in \cnt(B)$.
\end{proof}

The $\ZZ^n$-grading on $B$ extends naturally to $C$ and so does the notion of support. The ring $C$ is, trivially, a TGWA. Therefore, Lemma \ref{lem.simp1} and \ref{lem.cnt2} follow from corresponding results in \cite{HO}. However, as our notation is different, we include their short proofs here.

\begin{lemma}\label{lem.simp1}
Assume $R$ is $\Gamma$-simple.
Let $c \in C$ be nonzero. Then there exists an element $c' \in C$ satisfying:
\begin{enumerate}
\item $c' \in CcC$,
\item $(c')_\bzero=1$, and
\item $|\supp(c')| \leq |\supp(c)|$.
\end{enumerate}
\end{lemma}
\begin{proof}
Write $c$ according to the $\ZZ^n$-grading and choose $\balpha \in \ZZ^n$ such that $c_\balpha \neq 0$. Because $B$ (and hence $C$) is a domain, then we can replace $c$ by $ct^{-\balpha}$ so that we may assume that $c_\bzero \neq 0$. Now by Lemma \ref{lem.cnt1} (1), there exists a nonzero $r \in R = B_\bzero$ such that $rc \in B$. Since $B$ is a graded domain, then $rc_0 \neq 0$. 

Consider the set
\[ K = \left\lbrace s \in R ~\middle|~ s + \sum_{\balpha \in \supp(c) \backslash\{0\}} d_\balpha \in CcC \text{ for some }d_\balpha \in C_\balpha\right\rbrace.\]
Taking $d_\balpha = rb_\balpha$ for each $\balpha \in \supp(c) \backslash\{0\}$ shows that $rc_0 \in K$. Thus, $K$ is a nonzero ideal of $R$. 

We now show that $K$ is $\Gamma$-invariant. Let $s \in K$, so $s+\sum_{\balpha \in \supp(c) \backslash\{0\}} d_\balpha \in CcC$ for some $d_\balpha \in C_\balpha$. Then
\[
CcC \ni t_i \left(s+\sum_{\balpha \in \supp(c) \backslash\{0\}} d_\balpha\right)t_i\inv = \sigma_i(s) + \sum_{\balpha \in \supp(c) \backslash\{0\}} t_i d_\balpha t_i\inv.
\]
Clearly, $t_id_\balpha t_i\inv C_\balpha$ for each $\balpha$. This shows that $\sigma_i(s) \in K$. Similarly one shows that $\sigma_i\inv(s) \in K$. Thus, $K$ is a nonzero $\Gamma$-invariant ideal of $R$, so by $\Gamma$-simplicity, we have $K=R$. Thus, $1 \in K$, so there exists $d_\balpha \in C_\balpha$ such that $c':=1 + \sum_{\balpha \in \supp(c) \backslash\{0\}} d_\balpha \in CcC$. It is clear that $c'$ satisfies the three properties.
\end{proof}

We say a subalgebra $A'$ of an algebra $A$ is \emph{essential} if $A' \cap I \neq \{0\}$ for every nonzero ideal $I$ of $A$.
Since $C$ is a TGWA (of type $(A_1)^n$), then \cite[Theorem 7.8]{HO} immediately implies the following. (In that result, take $A$ to be $C$, so that the $B$ in that result is also $C$.)

\begin{lemma}\label{lem.cnt2}
Let $R$ be $\Gamma$-simple. Then $\cnt(C)$ is an essential subalgebra of $C$.
\end{lemma}
\begin{proof}
Let $K$ be a nonzero ideal of $C$. Choose $c \in K$ such that $c \neq 0$ and $|\supp(c)|$ is as small as possible. Let $c' \in CcC \subset K$ be the corresponding element determined by Lemma \ref{lem.simp1}. Note that by minimality we have $|\supp(c')|=|\supp(c)|$. 

Let $r \in R$. Then $(c'r-rc')_\bzero = r-r=0$, so $|\supp(c'r-rc')|<|\supp(c')|$. Since $c'r-rc' \in J$, then minimality implies that $c'r-rc'=0$. Choose $i \in \{1,\hdots,n\}$. Then $(t_ic't_i\inv - c')_0 = 1-1=0$, so the same argument as above implies that $t_ic'=c't_i$. It follows that $c' \in \cnt(C)$.
\end{proof}

The previous lemma shows a bit more. In particular, $c' \in K \cap \cnt(C) \cap \left(1+ \sum_{\balpha \in \ZZ^n\backslash\{0\}} C_\balpha\right)$.

\begin{lemma}\label{lem.simp2}
The BR algebra $B$ is simple if and only if the following conditions hold:
\begin{enumerate}
\item $C$ is simple
\item $BxB=B$ for all $x \in \cX$.
\end{enumerate}
\end{lemma}
\begin{proof}
Suppose that $B$ is simple. Then clearly (2) holds. Let $K$ be a nonzero ideal in $C$ and let $b \in K$ be nonzero. Then there exists $x \in \cX$ such that $xb \in B$. Then $0 \neq xb \in B \cap K$. Thus, $B \cap K$ is a nonzero ideal of $B$ and so $1 \in B \cap K \subset K$. It follows that $C$ is simple.

Now assume that the two conditions hold. Let $K$ be an ideal of $B$. Then $CKC$ is a two-sided ideal of $C$, so $1 \in CKC$. Now by \eqref{eq.ideal}, $1=x\inv a y\inv $ for some $x,y \in \cX$ and $a \in K$. That is, $xy = a \in K$. The second condition implies that $B=BxyB=BaB \subset K$. Thus $K=B$.
\end{proof}

\begin{lemma}\label{lem.cntequiv}
Consider the following:
\begin{enumerate}
\item \label{cnt1} $\cnt(B)$ is a field,
\item \label{cnt2} $\cnt(B) \subset R$, and
\item \label{cnt3} $\cnt(B) = R^\Gamma$.
\end{enumerate}
Then $\eqref{cnt1} \Rightarrow \eqref{cnt2} \Rightarrow \eqref{cnt3}$. If $R$ is $\Gamma$-simple, then all three are equivalent.
\end{lemma}
\begin{proof}
$\eqref{cnt1} \Rightarrow \eqref{cnt2}$
Suppose there exists some nonzero $a \in \cnt(B)_\balpha$ for some $\balpha \in \ZZ^n \backslash \{\bzero\}$. Then $1+a \in \cnt(B)$ and so there exists some $b \in \cnt(B)$ that is a multiplicative inverse for $1+a$. Let $<$ be an ordering on $\ZZ^n$. We may assume that $\balpha > 0$. Write $b = b_{\beta_1} + \cdots + b_{\beta_k}$ where $0 \neq b_{\beta_i} \in B_{\beta_i}$ with each $\beta_i \in \ZZ^n$. Assume $\beta_1 < \dots < \beta_k$. Hence, the term of lowest degree in $(1+a)b$ is $1b_{\beta_1}$ and the term of highest degree is $ab_{\beta_k} \neq 0$. It follows that $k=1$ and so $b_{\beta_1} + ab_{\beta_1}=1$ but this contradicts $\balpha+\beta_1>\beta_1$. Thus, $\cnt(A) \subset R$.

$\eqref{cnt2} \Rightarrow \eqref{cnt3}$
Let $r \in \cnt(B) \subset R$. Then for any $at_i \in B_{\be_i}$ we have $r(at_i)=(at_i)r$ or, equivalently, $(r-\sigma_i(r))at_i$ (here we use the fact that $R$ is commutative). Since $R$ is a domain, so is $B$ and hence $\sigma_i(r)=r$. 

Now assume that $R$ is $\Gamma$-simple. We claim $\eqref{cnt3} \Rightarrow \eqref{cnt1}$. Let $r \in \cnt(B)=R^\Gamma$. Then $Rr=rR$ is a $\Gamma$-invariant ideal of $R$. Thus, if $r\neq 0$, then $Rr=R$, so $r$ is invertible.
\end{proof}

\begin{lemma}\label{lem.simp3}
If $B$ is simple, then $R$ is $\Gamma$-simple.
\end{lemma}
\begin{proof}
Let $K$ be a proper $\Gamma$-invariant ideal of $R$. Since $R=B_0$, we have
\[
(BKB) \cap R = \sum_{\balpha \in \ZZ^n} B_{\balpha} K B_{-\balpha}
	= \sum_{\balpha \in \ZZ^n}\sigma^{\balpha}(K) B_{\balpha} B_{-\balpha}
	\subset KR \subset K.
\]
Hence, $BKB$ is a proper ideal of $B$, so $BKB=\{0\}$. That is, $K=\{0\}$.
\end{proof}

We are now ready to prove our simplicity criteria.

\begin{theorem}\label{thm.HBRsimple}
The BR algebra $B$ is simple if and only if the following conditions hold:
\begin{enumerate}
\item \label{spr1} $BxB=B$ for all $x \in \cX$,
\item \label{spr2} $R$ is $\Gamma$-simple, and
\item \label{spr3} $\cnt(B) \subset R$.
\end{enumerate}
\end{theorem}
\begin{proof}
Suppose $B$ is simple. Condition \eqref{spr1} is clear while
condition \eqref{spr2} follows from Lemma \ref{lem.simp3}. Since $B$ is simple, then $\cnt(B)$ is a field. Thus condition \eqref{spr3} follows from Lemma \ref{lem.cntequiv}.

Now suppose that the three conditions hold. Since we have assumed condition \eqref{spr1}, then by Lemma \ref{lem.simp2} it suffices to prove that $C$ is simple. Let $K$ be a nonzero ideal of $C$. By Lemma \ref{lem.cnt2}, $K \cap \cnt(C) \cap \left(1+ \sum_{\balpha \in \ZZ^n\backslash\{0\}} C_\balpha\right) \neq \emptyset$. Conditions \eqref{spr1} and \eqref{spr3} along with Lemma \ref{lem.cnt1}(2) imply that $\cnt(C)=\cnt(B) \subset R$. But then $\cnt(C) \cap \left(1+ \sum_{\balpha \in \ZZ^n\backslash\{0\}} C_\balpha\right)  = \{1\}$, so $1 \in K$.
\end{proof}

Unfortunately, Theorem \ref{thm.HBRsimple}~\eqref{spr1}
is difficult to check in practice. In what follows, we present
an alternative condition in the spirit of Proposition \ref{prop.simple}.
This condition is implied by Theorem \ref{thm.HBRsimple}~\eqref{spr1}
and we conjecture that it is equivalent in the context of that theorem.

\begin{lemma}
Let $i\in [n]$, then the following are equivalent:
\begin{enumerate}
    \item $\css_i(B)\cap \sigma_i^{k}(\css_i(B))=\emptyset$ for all $k\in\ZZ_{>0}$,
    \item $H_iJ_i+\sigma_i^k(H_iJ_i)=R$ for all $k\in\ZZ_{>0}$,
    \item $BI_i^{(k)}t_i^kB=B$ for all $k\in\ZZ_{>0}$.
\end{enumerate}
\end{lemma}
\begin{proof}
(1)$\iff$(2). We prove the contrapositives. First we observe that the following statements are equivalent:
\begin{itemize}
    \item $\css_i(B)\cap \sigma_i^{k}(\css_i(B))\neq\emptyset$,
    \item there are $\bp,\bq\in\Spec(R)$ such that $\bp=\sigma_i^k(\bq)$, $H_iJ_i\subset \bp$, and $H_iJ_i\subset \bq$, 
    \item $H_iJ_i\subset \bp$ and $\sigma_i^k(H_iJ_i)\subset \bp$, 
    \item $H_iJ_i+\sigma_i^k(H_iJ_i)\subset \bp$. 
\end{itemize}
Hence, if $\css_i(B)\cap \sigma_i^{k}(\css_i(B))\neq\emptyset$, then there is a prime ideal $\bp$ with $H_iJ_i+\sigma_i^k(H_iJ_i)\subset \bp\subsetneq R$. Conversely, if $H_iJ_i+\sigma_i^k(H_iJ_i)\subsetneq R$, then by Zorn's Lemma there is a maximal ideal $\frm\in\Spec(R)$ with $H_iJ_i+\sigma_i^k(H_iJ_i)\subset \frm$, hence $\css_i(B)\cap \sigma_i^{k}(\css_i(B))\neq\emptyset$.

(2)$\iff$(3). First, we claim that $BI_i^{(k)}t_i^kB\cap B_{(k-1)\be_i}=\sum_{\ell\in\ZZ}I_i^{(\ell)}t_i^\ell I_i^{(k)}t_i^k I_i^{(-\ell-1)}t_i^{-\ell-1}$. It is clear that we have the inclusion $\supset$, so suppose that $y \in BI_i^{(k)}t_i^kB\cap B_{(k-1)\be_i}$. Then, there exists $a_k \in I_i^{(k)}$, $b_\balpha \in I^{(\balpha)}$ and $c_{-(\balpha+\be_i)} \in I^{({-(\balpha+\be_i)})}$ for some $\balpha \in \ZZ^n$
such that
\begin{align*}
y &= \sum_{\balpha \in \ZZ^n} \left( b_\balpha \bt^\balpha \right)a_kt_i^k \left(c_{-(\balpha+\be_i)} \bt^{-(\balpha+\be_i)} \right) \\
&= \sum_{\balpha \in \ZZ^n} \gamma_\balpha (b_{\alpha_i} t_i^{\alpha_i}) \left( \sigma_i\inv(b_{\widehat{\balpha}}) \bt^{\balpha-\alpha_i\be_i} \right)a_kt_i^k \left(c_{-(\balpha+\be_i)} \bt^{-(\balpha+\be_i)} \right) \\
&= \sum_{\balpha \in \ZZ^n} \gamma_\balpha' (b_{\alpha_i} t_i^{\alpha_i}) \left(  (\sigma_i\inv(b_{\widehat{\balpha}}) \sigma^{\balpha-\alpha_i\be_i}(a_k))t_i^k \right)\left(\sigma^{\balpha-\alpha_i\be_i}(c_{-(\balpha+\be_i)}) t_i^{-\alpha_i - 1} \right)
\end{align*}
where $b_{\widehat{\balpha}}=b_{\alpha_i}\inv b_\balpha$ and some $\gamma_\balpha,\gamma_\balpha' \in \kk^\times$.
Now it is clear that $\gamma_\balpha' b_{\alpha_i} \in I_i^{(\alpha_i)}$, $\sigma_i\inv(b_{\widehat{\balpha}}) \sigma^{\balpha-\alpha_i \be_i}(a_k) \in I_i^{(k)}$, and $\sigma^{\balpha-\alpha_i\be_i}(c_{-(\balpha+\be_i)}) \in I_i^{(-\alpha_i - 1)}$. Setting $\alpha_i=\ell$ proves the claim.

Now, to simplify notation, fix $i$ and set $J=J_i$, $H=H_i$, $I=I_i$, $t=t_i$, and $\sigma=\sigma_i$. Then we have
\begin{align*}
BI^{(k)}t^kB\cap B_{(k-1)\be_i}
    &=\sum_{\ell\in\ZZ}I^{(\ell)}t^\ell I^{(k)}t^k I^{(-\ell-1)}t^{-\ell-1}\\
    &=\sum_{\ell\geq 1}I^{(\ell)}\sigma^{\ell}(I^{(k)})\sigma^{\ell+k}(I^{(-\ell-1)})t^{k-1}
        +RI^{(k)}\sigma^k(I^{(-1)})t^{k-1}+I^{(-1)}\sigma^{-1}(I^{(k)})t^{k-1}R \\
        &\qquad +\sum_{\ell\leq -2}I^{(\ell)}\sigma^{\ell}(I^{(k)})\sigma^{\ell+k}(I^{(-\ell-1)})t^{k-1}\\
    &=\sum_{\ell\geq 1}J\cdots \sigma^{\ell+k-1}(J)\sigma^{\ell+k}(\sigma^{-1}(H)\cdots \sigma^{-\ell-1}(H))t^{k-1}\\
        &\qquad + R J\cdots \sigma^{k-1}(J)\sigma^k\sigma^{-1}(H)t^{k-1}+R\sigma^{-1}(H)\sigma^{-1}(J\cdots \sigma^{k-1}(J))t^{k-1}\\
        &\qquad +\sum_{\ell\leq -2}\sigma^{-1}(H)\cdots \sigma^{\ell}(H)\sigma^{\ell}(J\cdots \sigma^{k-1}(J))\sigma^{\ell+k}(J\cdots \sigma^{-\ell-2}(J))t^{k-1}\\
    &=\sum_{\ell\geq 1}J\cdots \sigma^{\ell+k-1}(J)\sigma^{k-1}(H)\cdots \sigma^{\ell+k-1}(H))t^{k-1} \\
        &\qquad + R J\cdots \sigma^{k-1}(J)\sigma^{k-1}(H)t^{k-1}+R\sigma^{-1}(H)\sigma^{-1}(J)\cdots \sigma^{k-2}(J)t^{k-1} \\
        &\qquad +\sum_{\ell\leq -2}\sigma^{-1}(H)\cdots \sigma^{\ell}(H)\sigma^{\ell}(J)\cdots \sigma^{k-2}(J)t^{k-1}\\
    &= R\sigma^{k-1}(HJ) I^{(k-1)}t^{k-1}+R\sigma^{-1}(HJ) I^{(k-1)}t^{k-1}\\
    &= R(\sigma^{k-1}(HJ)+\sigma^{-1}(HJ)) I^{(k-1)}t^{k-1}.
\end{align*}
Since $\sigma^{k-1}(HJ)+\sigma^{-1}(HJ)=R$ if and only if $HJ+\sigma^k(HJ)=R$, this gives us the desired equivalence as in the proof of \cite[Lemma 7.19]{HO}.
\end{proof}

It is clear that the condition $BI_i^{(k)}t_i^kB=B$ for all $i\in [n]$, $k\in\ZZ_{>0}$ is implied by $BxB=B$ for all $x\in\cX$, but the reverse direction is not clear. 
Going forward, it would be good to develop further tools for analyzing the simplicity of BR algebras of rank $n>1$.

%\bibliography{biblio}{}
%\bibliographystyle{myplain}

\end{document}